\documentclass[reqno]{amsart}

\usepackage{amssymb}
\usepackage{amsmath}
\usepackage{IEEEtrantools}
\usepackage{amsthm}
\usepackage{mathtools}
\usepackage{mltex}
\usepackage{parskip}
\usepackage{enumerate}
\usepackage{hyperref}
\usepackage{tcolorbox}
\usepackage{graphicx}
\graphicspath{ {images/} }
\usepackage{amsmath,microtype}
\usepackage{multicol}
\usepackage{todonotes}
\usepackage{wrapfig}
\usepackage{tikz-cd}
\usepackage[english]{babel}
\usepackage [autostyle, english = american]{csquotes}
\MakeOuterQuote{"}
\usetikzlibrary{backgrounds}
\usepackage{dynkin-diagrams}
\usepackage{booktabs}
\usepackage{blkarray}
\usetikzlibrary{babel}
\usepackage{upgreek}
\usepackage{color,soul}

\DeclareMathOperator{\Ad}{Ad}
\DeclareMathOperator{\ad}{ad}

\DeclareMathOperator{\SO}{SO}
\DeclareMathOperator{\SU}{SU}
\DeclareMathOperator{\G}{G}
\DeclareMathOperator{\Sp}{Sp}
\DeclareMathOperator{\U}{U}

\DeclareMathOperator{\Spin}{Spin}

\DeclareMathOperator{\Id}{Id}
\DeclareMathOperator{\inv}{inv}
\DeclareMathOperator{\Span}{span}

\theoremstyle{plain}
\newtheorem{theorem}{Theorem}[section]
\newtheorem{lemma}[theorem]{Lemma}
\newtheorem{proposition}[theorem]{Proposition}

\newtheorem*{theorem*}{Theorem}

\theoremstyle{definition}
\newtheorem{definition}[theorem]{Definition}
\newtheorem{remark}[theorem]{Remark}

\newcommand{\Z}{\mathbb{Z}}

\newcommand{\R}{\mathbb{R}}
\newcommand{\C}{\mathbb{C}}

\newcommand{\tad}{{\text{$3$-}(\alpha,\delta)\text{-Sasaki}}}
\newcommand{\ts}{3\text{-Sasakian}}
\renewcommand{\H}{\mathbb{H}}

\newcommand{\RP}{\mathbb{RP}}
\renewcommand{\:}{\colon}
\newcommand{\vect}{\mathfrak{X}}

\newcommand{\Lie}{\mathcal{L}}

\usepackage[margin=3cm]{geometry}

\newtcolorbox{mybox}
{colframe = purple!25,
  colback  = purple!10,
  coltitle = purple!20!black,  
  title    = Idea}
  
\newtcolorbox{mybox2}
{colframe = red!75,
  colback  = red!10,
  coltitle = red!20!white,  
  title    = Problem}

\newtcolorbox{mybox3}
{colframe = blue!75,
  colback  = blue!10,
  coltitle = blue!20!white,  
  title    = To be completed/added later}  
  
\begin{document}

\title{$\mathcal{H}$-Killing Spinors and Spinorial Duality for Homogeneous 3-$(\alpha,\delta)$-Sasaki Manifolds}
\author{Ilka Agricola and Jordan Hofmann}
\begin{abstract}
We show that $\tad$ manifolds admit solutions of a certain new spinorial field equation (the \emph{$\mathcal{H}$-Killing equation}) generalizing the well-known Killing spinors on $\ts$ manifolds. These $\mathcal{H}$-Killing spinors have more desirable geometric properties than the spinors obtained by simply deforming a $\ts$ metric; in particular we obtain a one-to-one correspondence between $\mathcal{H}$-Killing spinors on dual pairs of homogeneous $\tad$ spaces. Finally, we show that $\mathcal{H}$-Killing spinors generalize certain special spinors in dimension $7$ previously constructed by Agricola-Friedrich and Agricola-Dileo using $\G_2$-geometry. 
\end{abstract}

\maketitle




\tableofcontents
\section{Introduction}

One of the most consequential mathematical achievements of the past century is the Berger holonomy theorem \cite{Berger_holonomy,Simons_holonomy}, which classifies the possible holonomy groups of irreducible non-symmetric Riemannian manifolds. The only odd-dimensional geometry appearing in Berger's list is $\G_2$ in dimension $7$, which has led over the subsequent decades to the study of \emph{non-integrable} odd-dimensional geometries via connections better adapted to the geometry: so-called \emph{characteristic connections} (see e.g.\@ \cite{SRNI}). In dimension $3$(mod $4$) in particular, natural and abundant examples of non-integrable geometries are given by $\ts$ manifolds, which, roughly speaking, are characterized by the existence of certain tensors facilitating the identification of the tangent spaces with $\text{Im}(\H)\oplus \H^n$. These are related to two of the even-dimensional geometries from Berger's list; the cone over a $\ts$ manifold is hyperK\"{a}hler (in fact one can take this as a definition), and, if the \emph{Reeb vector fields} are complete, the space of leaves of the corresponding $3$-dimensional foliation is a quaternionic K\"{a}hler orbifold \cite{BG3Sas}. 


It is well-known that every $\ts$ manifold is spin (as a consequence of the reduction of structure group of the tangent bundle proved by Kuo in \cite{3Sas_structure_reduction}), and they are highly relevant objects in spin geometry, carrying many examples of interesting spinor fields (see e.g.\@ \cite{Bar,3Sasdim7}). In fact, B\"{a}r showed in \cite{Bar} that the existence of a compatible 3-Sasakian structure on a Riemannian manifold is completely determined by the number of linearly independent Riemannian Killing spinors carried by the metric, i.e.\@ the number of linearly independent spinors satisfying the \emph{Killing equation}
\[
\nabla^g_X\psi = \pm \frac{1}{2} X\cdot \psi \qquad \text{ for all } X\in TM.
\]Specifically, B\"{a}r proved that $(M^{4n-1},g)$ carries a compatible 3-Sasakian structure if and only if there exist at least $n+1$ linearly independent Riemannian Killing spinors for the Killing number $\frac{1}{2}$ (or $-\frac{1}{2}$, up to a change of orientation). The existence of Killing spinors puts strong constraints on the geometry of the manifold, implying, for example, that the metric is Einstein and that equality is achieved in Friedrich's estimate for the first eigenvalue of the Dirac operator (see e.g.\@ \cite[Chaps.\@ 5.1, 5.2]{FriedrichBook}). The first authors to realize that spinors apart from Killing spinors can be deformed and carry important information were Friedrich and Kim, who defined in \cite{FrKi00} the class of \emph{Sasakian quasi-Killing spinors} and proved that they induce, in certain dimensions, solutions of the coupled Einstein-Dirac system. Other generalizations of Killing spinors have also been studied, and the existence of such spinors is known to be closely related to $\text{G}$-structures in low dimensions (see e.g.\@ \cite{Spin7,hypo,dim67}).


The aim of the present article is to identify spinors satisfying new types of remarkable field equations on $3$-contact manifolds of every dimension $4n-1\geq 7$, outside of the well-studied $\ts$ case. The setting in which we do this is that of $\tad$ structures, using their remarkable geometric properties such as the existence of a \emph{canonical connection} (a metric connection with totally skew-symmetric torsion well-adapted to the $3$-contact structure). The strategy we employ is a modification of Friedrich and Kath's method in \cite{Fried90}; we identify a certain spinorial connection which acts trivially on a subbundle of the spinor bundle and then extract a spinorial equation by writing this connection as the sum of the Levi-Civita connection and another term. Furthermore, in the homogeneous setting we show that our spinors are compatible with a certain natural notion of duality, observed in \cite{hom3alphadelta}, between positive and negative $\tad$ homogeneous spaces. The study of homogeneous $\tad$ spaces is heavily rooted in Boyer, Galicki, and Mann's classification of homogeneous $\ts$ spaces in \cite{BG3Sas}, and we note that a new, self-contained, and elegant Lie theoretic proof of this classification has recently appeared in \cite{Leander_roots}. 

%
%

Our results are as follows:

\begin{itemize}
\item In \S\ref{section:actual_deformedKS} we describe how the Killing spinor equation on a $\ts$ manifold changes under so-called \emph{$\mathcal{H}$-homothetic deformations}, which transform the $\ts$ structure into a $\tad$ structure with $\alpha\delta>0$. We obtain Theorem \ref{deformed_KS_theorem_ref}, which asserts that a positive $\tad$ manifold of dimension $4n-1$ carries at least $n+1$ linearly independent solutions of the \emph{deformed Killing equation} (\ref{KS_with_torsion}).
\item The so-called \emph{deformed Killing spinors} do not naturally extend to negative $\tad$ structures ($\alpha\delta<0$), so in \S\ref{section:horizontalKS} we prove the existence of spinors satisfying a different generalization of the Killing equation which is valid for any $\tad$ manifold (Theorem \ref{deformedKillingspinorstheorem}). We call these \emph{$\mathcal{H}$-Killing spinors} to emphasize that they behave as Killing spinors along the horizontal distribution $\mathcal{H}$. Along the vertical distribution they satisfy the Killing equation plus a correction term which vanishes in the $3$-$\alpha$-Sasaki case ($\alpha=\delta$).
\item Section \ref{section:duality} is a bit different in flavour, as it concentrates on the homogeneous case, where more structure and a full classification (for $\alpha\delta>0$) are available \cite{hom3alphadelta}. In this section we demonstrate that the $\mathcal{H}$-Killing equation is a natural one to consider by proving that $\mathcal{H}$-Killing spinors on a compact homogeneous $\tad$ space and its non-compact dual correspond in a one-to-one manner.
\item Finally, in \S\ref{section:dim7} we use certain algebraic identities in dimension $7$ to prove in Theorem \ref{Chap_duality:dim7_equivalence_of_two_eqns} that the $\mathcal{H}$-Killing equation in this dimension is equivalent to the generalized Killing equation satisfied by the \emph{auxiliary spinors} introduced in \cite{3Sasdim7,3str}. The auxiliary spinors are defined using the natural $\G_2$-structure induced by a $\tad$ structure in dimension $7$ (see \cite[Thm.\@ 4.5.1]{3str}), and therefore until now had no obvious generalization to higher dimensions.
\end{itemize}

\textbf{Acknowledgments:} The second author was supported by the Engineering and Physical Sciences Research Council [EP/L015234/1, EP/W522429/1]; the London Mathematical Society [ECR-2223-88]; the EPSRC Centre for Doctoral Training in Geometry and Number Theory (The London School of Geometry and Number Theory); University College London; King's College London; and Imperial College London. Parts of this work have previously appeared in the second author's PhD thesis.

\section{Preliminaries}
For a detailed introduction to spin geometry we refer to \cite{LM,FriedrichBook}, and for 3-Sasakian and $\tad$ geometry we recommend \cite{BG3Sas,3Sasdim7,3str}.

\subsection{3-($\alpha,\delta$)-Sasaki Structures}\label{section:tad_definitions}
Here we recall the basic definitions related to $\tad$ manifolds, following the notation and conventions of \cite{3str}. For systematic examination of these structures we refer the reader to \cite{3str}.

\begin{definition}(Almost 3-contact, almost 3-contact metric, 3-Sasakian). An almost 3-contact manifold $(M^{4n-1},\varphi_i,\xi_i,\eta_i)_{i=1}^3$ consists of a manifold of dimension $4n-1$ together with three almost contact structures satisfying 
	\begin{align*}
	\varphi_i(\xi_j)&=-\varphi_j(\xi_i)= \xi_k, \qquad \eta_i=\eta_j\circ \varphi_k = -\eta_k\circ\varphi_j,\\
		\varphi_i& = \varphi_j\circ \varphi_k - \eta_k\otimes \xi_j = -\varphi_k\circ \varphi_j + \eta_j\otimes \xi_k  
	\end{align*}
	for all even permutations $(i,j,k)$ of $(1, 2, 3)$. The structure tensors induce a splitting $TM=\mathcal{H}\oplus \mathcal{V}$ into \emph{horizontal} and \emph{vertical} spaces, defined by 
	\begin{align*} \mathcal{H}:= \cap_{i=1}^3 \ker(\eta_i),\qquad \mathcal{V}:= \text{span}_{\R}\{ \xi_i\}_{i=1}^3,
	\end{align*}and the horizontal restrictions $\varphi_i^{\mathcal{H}}:= \varphi_i\rvert_{\mathcal{H}}$ satisfy the defining relations of the quaternions: \[\varphi_i^{\mathcal{H}} \circ  \varphi_j^{\mathcal{H}}=\varphi_k^{\mathcal{H}}\] for all even permutations $(i,j,k)$ of $(1,2,3)$. In particular this defines a quaternionic contact structure. The action of $\varphi_1$, $\varphi_2$, $\varphi_3$ in the vertical directions is determined by
\begin{align}
	\label{phii_vertical_directions} \varphi_i(\xi_j)=\xi_k 
\end{align}
for all even permutations $(i,j,k)$ of $(1,2,3)$. Notably, any almost 3-contact manifold admits a single metric which is compatible with all three almost contact structures in the sense that
	\begin{align*} g(\varphi_i( - ),\varphi_i( - )) = g - \eta_i\otimes \eta_i \quad \text{ for all } i=1,2,3 , 
		\end{align*}
and we call the data $(M,g,\varphi_i,\xi_i,\eta_i)$ an almost 3-contact metric manifold. If, furthermore, each of the three structures is Sasakian, the manifold is called 3-Sasakian.
\end{definition}
With the above definitions in mind, $\tad$ structures are defined as almost $3$-contact metric structures satisfying a certain integrability condition:
\begin{definition}
	A 3-$(\alpha,\delta)$-Sasaki manifold is an almost 3-contact metric manifold $(M,g,\varphi_i,\xi_i,\eta_i)$ satisfying $$d\eta_i = 2\alpha\Phi_i+2(\alpha-\delta)\eta_j\wedge \eta_k , \qquad \alpha,\delta\in \R, \ \alpha\neq 0$$ for all even permutations $(i,j,k)$ of $(1,2,3)$, where $\Phi_i := g(  - , \varphi_i(-))$ is the metric dual (in the second component) of the skew-symmetric endomorphism $\varphi_i$.
\end{definition}
This notion contains the 3-Sasakian spaces ($\alpha=\delta=1$) as well as many other interesting classes such as Einstein 3-$\alpha$-Sasakian structures ($\alpha=\delta$), parallel structures ($\delta=2\alpha$), degenerate structures ($\delta=0$), and a second Einstein metric ($\delta=(2n+1)\alpha$, where $\dim M:= 4n-1$).

\subsection{Connections on $\tad$ Manifolds}
In this subsection we review the canonical connection of a $\tad$ manifold, and, in the homogeneous setting, the algebraic description of the canonical and Levi-Civita connections in terms of \emph{Nomizu maps}. To begin, we recall from \cite{3str} the formulas for the Levi-Civita derivatives of the structure tensors:
\begin{proposition}\label{LCder} \emph{(\cite[Prop.\@ 2.3.2, Cor.\@ 2.3.1]{3str}).}
	If $(M,g,\xi_i,\eta_i,\varphi_i)$ is a $\tad$ manifold then
	\begin{align}
		(\nabla_Y^g \varphi_i)X &=\alpha[ g(X,Y) \xi_i - \eta_i(X)Y]-2(\alpha-\delta)[\eta_k(Y)\varphi_jX - \eta_j(Y)\varphi_kX] \\
		&\qquad \qquad +(\alpha-\delta)[\eta_j(Y)\eta_j(X) +\eta_k(Y)\eta_k(X)] \xi_i  \nonumber \\
		&\qquad \qquad -(\alpha-\delta)\eta_i(X) [\eta_j(Y)\xi_j +\eta_k(Y)\xi_k] , \nonumber \\
		\nabla_Y^g \xi_i &=-\alpha \varphi_i(Y) - (\alpha-\delta)[\eta_k(Y)\xi_j -\eta_j(Y)\xi_k] \label{LC_derivative_xii_formula} ,
	\end{align}
	for any even permutation $(i,j,k)$ of $(1,2,3)$.
\end{proposition}

Indeed, these formulas become rather unwieldy when $\alpha\neq \delta$, and thus the Levi-Civita connection is not well suited to computations in the $\tad$ setting. While it is not possible in general to find a nice connection parallelizing all the structure tensors of a 3-$(\alpha,\delta)$-Sasaki manifold (see \cite[Thm.\@ 4.2.1]{3str}), there is nonetheless a good choice available:

\begin{proposition}\label{prelims:canonical_connection}\emph{(Thms.\@ 4.1.1, 4.4.1, and the discussion in Section 4 of \cite{3str}).}
	Let $(M,g,\varphi_i,\xi_i,\eta_i)$ be a 3-$(\alpha,\delta)$-Sasaki manifold. Then $M$ admits a unique metric connection $\nabla$ with skew torsion such that, for some function $\beta\in C^{\infty}(M)$, $$\nabla_X\varphi_i = \beta ( \eta_k(X) \varphi_j -\eta_j(X) \varphi_k) \quad \forall X\in \Gamma(TM). $$
	Furthermore, $\beta$ is the constant function $\beta\equiv 2(\delta-2\alpha)$ and the torsion of $\nabla$ is the 3-form \begin{align} \label{canonical_torsion_formula} T= \sum_{i=1}^3 \eta_i\wedge d\eta_i + 8(\delta-\alpha) \eta_{1,2,3} = 2\alpha \sum_{i=1}^3 \eta_i\wedge \Phi_i^{\mathcal{H}} +2(\delta - 4\alpha) \eta_{1,2,3}  ,\end{align} 
	where $\eta_{1,2,3}:=\eta_1\wedge \eta_2\wedge \eta_3$. The $\nabla$-derivatives of the structure tensors are given by 
	\begin{align*}
	\nabla_X\varphi_i &= \beta ( \eta_k(X) \varphi_j -\eta_j(X) \varphi_k), \quad
	\nabla_X\xi_i = \beta ( \eta_k(X) \xi_j -\eta_j(X) \xi_k), \quad
	\nabla_X\eta_i = \beta ( \eta_k(X) \eta_j -\eta_j(X) \eta_k),
	\end{align*}
	and, in particular, are parallel in the horizontal directions.
\end{proposition}
 As a matter of notation, throughout the article we shall always denote by $T$ the torsion of $\nabla$. The connection $\nabla$ is called the \emph{canonical connection} of the $\tad$ structure, and we shall make frequent use of it to simplify the curvature calculations in \S\ref{section:horizontalKS}. We will also sometimes use the shorthand $\eta_{i,j}:=\eta_i\wedge \eta_j$, $e_{i,j}:=e_i\wedge e_j$, etc.\@ for differential forms. To translate curvature identities for the canonical connection back in terms of the Levi-Civita connection, we need the following:
\begin{proposition}
\emph{(\cite[Section 1.2]{ADScurv}).} The curvature tensors $R, R^g$ of the canonical and Levi-Civita connections are related by
\begin{align}R^g(X,Y,Z,V) = R(X,Y,Z,V) -\frac{1}{4} g(T(X,Y),T(Z,V)) -\frac{1}{8}dT(X,Y,Z,V) ,  \label{curvdifference}
\end{align}
where $T$ is the torsion tensor of $\nabla$.
\end{proposition}
Finally we recall the description of metric connections in the homogeneous setting via Nomizu maps. Indeed, it is well-known (see e.g. \cite{Nomizumap,Wangconnections,KN1}) that an invariant metric connection $\nabla'$ (we use $\nabla'$ rather than $\nabla$ in order to distinguish it from the canonical connection defined above) on a Riemannian homogeneous space $(M=G/H, g)$ corresponds to an $H$-equivariant linear map $T_oM \to \mathfrak{so}(T_oM)$ ($o:=eH$). Following the modern exposition in \cite[Chap.\@ 6]{ANT_book_principal_fibre_bundles}, this relationship may be expressed as follows:
\begin{proposition}Let $(M=G/H,g)$ be a Riemannian homogeneous space with reductive decomposition $\mathfrak{g}=\mathfrak{h} \oplus \mathfrak{m}$, and let $(\pi_*)_e \: \mathfrak{m}\to T_oM$ ($o:=eH$) denote the isomorphism induced by the (restriction to $\mathfrak{m}$ of the) derivative of the projection $\pi \: G\to M=G/H$ at the identity element.
	\begin{enumerate}[(i)] \item There is a one-to-one correspondence between $G$-invariant metric connections $\nabla'$ and linear maps $\Uplambda' \: \mathfrak{m}\to \mathfrak{so}(\mathfrak{m})$ satisfying
	\[ 
	\Uplambda'(h\cdot X)Y = h\cdot (\Uplambda'(X)(h^{-1}\cdot Y))  \qquad \text{ for all } X,Y\in \mathfrak{m}, \ h\in H,
	\]
	where $h\cdot V := \Ad_H(h)V$ denotes the isotropy action of $h\in H$ on $V\in \mathfrak{m}\cong T_oM$. 
	\item The correspondence in (i) is given explicitly by 
	\[
	(\pi_*)_e (\Uplambda'(X)Y )= (\nabla'_{\widehat{X}} - \Lie_{\widehat{X}})_o(\widehat{Y}_o)   \qquad \text{ for all } X,Y\in \mathfrak{m},
	\]
	where $\widehat{X},\widehat{Y}$ denote the fundamental vector fields induced by $X,Y$.
	\item If $\tau$ is a $G$-invariant tensor field or differential form on $M$, then 
	\[
	(\nabla'_{\widehat{X}}\tau)_o= \Uplambda'(X)\tau_o \qquad \text{ for all } X\in\mathfrak{m} , 
	\] 
	where $\widehat{X}$ denotes the fundamental vector field induced by $X$ and the action of $\Uplambda'(X)\in \mathfrak{so}(\mathfrak{m})$ on $\tau$ is via the extension to the tensor or exterior algebra of the representation of $\mathfrak{so}(\mathfrak{m})$ on $\mathfrak{m}$. 
	\end{enumerate}
\end{proposition}
In particular, statement (iii) in the preceding proposition allows one to compute covariant derivatives of invariant tensors and differential forms in a purely algebraic way. To similarly compute the covariant derivatives of invariant spinors (with respect to the spinorial connection $\widetilde{\nabla'}$ induced by $\nabla'$), one uses the lifted Nomizu map $\widetilde{\Uplambda'}\: \mathfrak{m} \to \mathfrak{spin}(\mathfrak{m})$ obtained by composing $\Uplambda$ with the Lie algebra isomorphism $\mathfrak{so}(\mathfrak{m})\cong \mathfrak{spin}(\mathfrak{m})$. The relationship between $\widetilde{\nabla'}$ and $\widetilde{\Uplambda'}$ is then given by a similar formula as in (iii) above, with the action of $\widetilde{\Uplambda'}(X)\in \mathfrak{spin}(\mathfrak{m})$ on an invariant spinor $\psi_o \in (\Sigma_{\inv})_o$ via the spin representation. It is well-known that the Nomizu map of the Levi-Civita connection is given by a simple formula in terms of the Lie bracket of $\mathfrak{g}$:
\begin{proposition}\label{prelims:LC_Nomizu_general_formula}
	\emph{(\cite{ANT_book_principal_fibre_bundles,KN2,Nomizumap}).} If $(M=G/H,g)$ is a Riemannian homogeneous space then the Nomizu map $\Uplambda^g$ associated to the Levi-Civita connection is given by
	\[
	\Uplambda^g(X)Y = \frac{1}{2} \text{\emph{proj}}_{\mathfrak{m}}[X,Y] + U(X,Y)
	\]for all $X,Y\in\mathfrak{m}$, where $U$ is the symmetric $(2,0)$-tensor specified by 
	\begin{align}\label{Utensor}
		2g(U(X,Y),Z) = g(\text{\emph{proj}}_{\mathfrak{m}} [Z,X],Y) + g(X,\text{\emph{proj}}_{\mathfrak{m}}[Z,Y]).
	\end{align}
\end{proposition}

We conclude by recalling from \cite{hom3alphadelta} the explicit formulas for the Nomizu maps of the canonical and Levi-Civita connections of a homogeneous $\tad$ space. We give the formulas only in the case of a $\tad$ homogeneous space fibering over a Wolf space (i.e.\@ a symmetric quaternionic K\"{a}hler space), which are much simpler than the general case and will be sufficient for our purposes.
\begin{proposition}\emph{(\cite[Props.\@ 4.2.1, 4.2.2]{hom3alphadelta}).} Let $(M=G/H,g,\xi_i,\eta_i,\varphi_i)$ be a $\tad$ homogeneous space fibering over a Wolf space, and let $\mathfrak{g}=\mathfrak{h}\oplus \mathfrak{m}$ be a reductive decomposition. Denote by $\mathcal{V},\mathcal{H} \subseteq \mathfrak{m} $ the respective images under the identification $T_oM\cong \mathfrak{m}$ of the vertical and horizontal spaces at the origin.
	\begin{enumerate}[(i)]
		\item The Nomizu map for the canonical connection is given by
\begin{align}\label{canonical_Nomizu_map_explicit_formula}
	\Uplambda(V)W=
\begin{cases}
	\frac{\delta-2\alpha}{\delta}[V,W]& V \in \mathcal{V},\\
	0 & V\in \mathcal{H},
\end{cases} 
\end{align}
	\item The Nomizu map for the Levi-Civita connection is given by
\begin{align}\label{LC_Nomizu_map_explicit_formula}
	\Uplambda^g(V)W=
	\begin{cases}
		\frac{1}{2}\text{\emph{proj}}_{\mathfrak{m}}[V,W]&  V,W\in \mathcal{V} \text{\emph{ or }} V,W\in\mathcal{H}\\
		\left(1-\frac{\alpha}{\delta} \right)[V,W]&  V\in \mathcal{V}, \ W\in\mathcal{H}\\
		\frac{\alpha}{\delta}[V,W]&  V\in\mathcal{H}, \ W\in\mathcal{V}. \\
	\end{cases}	
\end{align}
\end{enumerate}
\end{proposition} 
Note that the choice of reductive complement $\mathfrak{m}$ in the preceding proposition is unique for the following reason. Any $\tad$ homogeneous space fibering over a Wolf space is specified by certain Lie algebraic data, called \emph{generalized $\ts$ data}, described in \cite{hom3alphadelta}. In particular the Lie algebra $\mathfrak{g}$ is simple, hence the complement $\mathfrak{m}$ is canonically given by the orthogonal complement of the isotropy subalgebra $\mathfrak{h}\subseteq \mathfrak{g}$ with respect to the Killing form. Explicit realizations of the reductive complement $\mathfrak{m}$ are given in \cite{homdata} in the case of homogeneous $\ts$ spaces, and this description carries over to the $\tad$ setting (see \cite{hom3alphadelta}).

\subsection{Spin Structures and the Spinor Bundle}

In this subsection we briefly recall the basic definitions and facts about spinors, following the conventions of \cite{AHLspheres}. Let $(M^m,g)$ be an oriented $m$-dimensional Riemannian manifold and $P_{\SO}$ its oriented orthonormal frame bundle. Furthermore, let $\lambda \: \Spin(m)\to \SO(m)$ be the double covering homomorphism and by abuse of notation let $R\: P_G \times G \to P_G $ denote right multiplication map ($G=\SO(m), \Spin(m)$). A \emph{spin structure} on $M$ is a principal $\Spin(m)$-bundle $P_{\Spin}$ together with a $2$-to-$1$ covering map $\Lambda \: P_{\Spin}\to P_{\SO}$ such that the following diagram commutes:

\[\begin{tikzcd}
	{P_{\Spin}\times \Spin(m)} & {P_{\Spin}} \\
	&& M \\
	{P_{\SO}\times \SO(m)} & {P_{\SO}}
	\arrow[from=1-2, to=2-3]
	\arrow[from=3-2, to=2-3]
	\arrow[from=1-1, to=1-2, "R"]
	\arrow[from=3-1, to=3-2, "R"]
	\arrow["\lambda",from=1-2, to=3-2]
	\arrow["\Lambda \times \lambda ",from=1-1, to=3-1]
\end{tikzcd}\]

The \emph{spinor bundle} is the associated vector bundle
\[
\Sigma M := P_{\Spin}\times_{\Spin(m)} \Sigma 
\]
induced by the \emph{spin representation} $\Spin(m)\to \text{GL}(\Sigma)$. The Levi-Civita connection $\nabla^g$ induces a connection on $\Sigma M$, which we also denote by $\nabla^g$ (by composing the connection $1$-form with the Lie algebra isomorphism $\mathfrak{so}(m)\cong \mathfrak{spin}(m)$), and this spinorial connection parallelizes a $\Spin(m)$-invariant Hermitian product $\langle \ , \ \rangle$.

The following is an explicit realization of the spin representation in dimension $m=4n-1$. Letting $e_1,\dots, e_{4n-1}$ be the standard orthonormal basis of $(\R^{4n-1},g_{\text{Euc}})$, define
\[
L' :=  \Span_{\C} \{ x_j:= \frac{1}{\sqrt{2}}( e_{2j}-ie_{2j+1})  \}_{j=1}^{2n-1}, \qquad L' :=  \Span_{\C} \{ y_j:= \frac{1}{\sqrt{2}}( e_{2j}+ie_{2j+1})  \}_{j=1}^{2n-1}
\]
to be respectively the spaces of formal $(1,0)$ and $(0,1)$ vectors in $T^{\C}M$ (we do not assume the existence of a complex/contact structure for this definition). The spin representation of \begin{align} \label{spin_group_def}\Spin(m)  := \{ v_1  \dots v_{2\ell} \: v_i\in \R^n, ||v_i||=1, \ell \in \mathbb{Z} \} \subset \langle \R^n  \ \rvert \  v\cdot w + w\cdot v = -2g_{\text{Euc}}(v,w) \rangle  \end{align} may be realized on the vector space $\Sigma:= \Lambda^{\bullet}L'$ by specifying the action of the unit vectors $e_i$:
\begin{align}\label{cliffordmultONB}
	e_{2j}\cdot \eta  &=  i(x_j \lrcorner \eta + y_j\wedge \eta ),\quad e_{2j+1}\cdot \eta =  (y_j\wedge\eta - x_j\lrcorner\eta ),\quad 
	e_{1}=  i\Id\rvert_{\Sigma^{\text{even}}} -i\Id\rvert_{\Sigma^{\text{odd}}},
\end{align}
for all $\eta\in \Sigma$. The algebra on the right hand side of (\ref{spin_group_def}) is called the \emph{Clifford algebra}, and the action of $\R^n$ on $\Sigma$ by the above formulas is called \emph{Clifford multiplication}. Clifford multiplication extends naturally to an action of $\Lambda^{\bullet}\R^m$ on $\Sigma$ by specifying $e_{i_1}\wedge \dots \wedge e_{i_\ell} \cdot \eta = e_{i_1}\cdot \ldots \cdot e_{i_\ell} \cdot \eta$ for all $\eta \in \Sigma$, and at the level of bundles this gives an action of vector fields and differential forms on spinors. Furthermore, Clifford multiplication by vector fields is skew-symmetric with respect to the Hermitian form $\langle \ , \ \rangle$.

\subsection{Killing Spinors on $\ts$ Manifolds} In this subsection we recall previously known results about Killing spinors on $\ts$ manifolds. Throughout the article we shall always fix the orientation such that all Killing spinors correspond to the Killing number $\frac{1}{2}$ (which is possible in dimension 3 (mod 4)). The story begins in the late 80s and early 90s with results obtained independently by several different authors. Focusing specifically on the Sasakian and $\ts$ cases, Friedrich and Kath constructed in \cite{Fried90} certain rank $2$ subbundles of the spinor bundle carrying a basis of Killing spinors:
\begin{theorem}\emph{(\cite[Thm.\@ 1]{Fried90}).}
	If $(M,g,\xi,\eta,\varphi)$ is an Einstein-Sasakian manifold, then
	\begin{align*}
		E:=  \{ \psi\in\Gamma(\Sigma M): (-2\varphi(X) + \xi\cdot X -X\cdot \xi)\cdot \psi =0 \ \text{ for all } X\in TM\} 
	\end{align*}
	has rank $2$ and is spanned over $C^{\infty}(M)$ by Killing spinors for the Killing number $\frac{1}{2}$.  
\end{theorem}

\begin{remark}\label{Ei_bundle_remark}
	For a $\ts$ manifold $(M,g,\xi_i,\eta_i,\varphi_i)$, which is automatically spin and Einstein (see e.g.\@ \cite[Cor.\@ 2.2.4, Cor.\@ 2.2.5]{BG_3Sas_paper}), one can consider for each individual Sasakian structure the bundle defined in the preceding theorem. This gives three rank $2$ bundles
	\begin{align}\label{Ei_definition}
	E_i:=  \{ \psi\in\Gamma(\Sigma M): (-2\varphi_i(X) + \xi_i\cdot X -X\cdot \xi_i)\cdot \psi =0 \ \text{ for all } X\in TM\} 
\end{align}
spanned by Riemannian Killing spinors. In Section \ref{section:horizontalKS} we will search for candidates for interesting spinors in the (non-direct) sum $E:=E_1+E_2+E_3$, whose rank satisfies $2\leq \text{rank}(E) \leq 6$. The failure of the sum to be direct can be seen in some examples in Table \ref{Tab:spinors_PsiEi_low_dim}, where we recall from \cite{Hof22} explicit bases for the $E_i$ in some low dimensions.
\end{remark}

In the preceding remark we note that the bundles $E_i$ associated to a $\ts$ manifold are spanned by at least $2$ and at most $6$ linearly independent Killing spinors. In dimension $7$, combining results from \cite{Fried90,BFGK,nearly_parallel_g2} gives a full description of the geometric structures carried by a spin manifold according to the number of linearly independent Killing spinors:
\begin{theorem}\emph{(\cite{Fried90,nearly_parallel_g2}).}
Let	$(M^7,g)$ be a $7$-dimensional spin manifold.
\begin{enumerate}[(i)]
	\item $(M^7,g)$ carries at least three linearly independent Killing spinors if and only if it admits a $\ts$ structure;
	\item $(M^7,g)$ carries at least two linearly independent Killing spinors if and only if it admits an Einstein-Sasakian structure;
	\item $(M^7,g)$ carries a non-trivial Killing spinor if and only if it admits a nearly parallel $\G_2$-structure. 
\end{enumerate}
\end{theorem}
Separately, the Killing spinors on $\ts$ manifolds have been studied from the point of view of holonomy. The crucial ingredient in this approach is the following result of B\"{a}r:
\begin{theorem}\label{prelims:Bar_cone_Thm}\emph{(\cite{Bar}).}
	There is a correspondence between real Killing spinors on a Riemannian spin manifold $(M,g)$ and parallel spinors on its cone $(\overline{M}:=M\times \R, \bar{g}:= r^2 g+ dr^2)$. 
\end{theorem}
In \S\ref{section:tad_definitions} we recalled the definition of $\ts$ manifolds in terms of the structure tensors $(g,\xi_i,\eta_i,\varphi_i)$, however $\ts$ manifolds may be equivalently defined by the condition that their cone is hyperK\"{a}hler (see \cite{BG_3Sas_paper}). Combining this with Wang's earlier result about the dimensions of the spaces of parallel spinors carried by the geometries from Berger's list (\cite{Wang}), B\"{a}r obtained:
\begin{theorem}\emph{(\cite[Thm.\@ 4']{Bar}).}
	Let $(M^{4n-1},g)$ be a complete simply-connected Riemannian spin manifold of dimension $4n-1>7$. Then $(M,g)$ admits a $\ts$ structure if and only if it carries at least $n+1$ linearly independent Killing spinors.
\end{theorem}
Note that this is a slight rephrasing of \cite[Thm.\@ 4']{Bar}, where we have used the fact that the spinor bundle of the round sphere is trivialized by Killing spinors (by Theorem \ref{prelims:Bar_cone_Thm} together with the fact that the cone over the round sphere is the flat Euclidean space). Recently, the investigation of the relationship between Killing spinors and Sasakian and $\ts$ structures was continued by the second author in \cite{Hof22}, and the following explicit formulas were obtained in the homogeneous setting:
\begin{theorem}\label{prelims:inv_KS_theorem}\emph{(\cite[Thm.\@ 6.1, Prop.\@ 6.4]{Hof22}).}
	Let $(M^{4n-1}=G/H,g,\xi_i,\eta_i,\varphi_i)$ be a simply-connected homogeneous $\ts$ space, and fix a realization of the spin representation with respect to an adapted basis of the reductive complement $\mathfrak{m}$.
	\begin{enumerate}[(i)]
		\item If $n\geq 2$, the space of invariant Killing spinors on $(M,g)$ has a basis given by
		 \begin{align*}
		 \psi_k:= \omega^{k+1} -i(k+1) y_1\wedge \omega^k, \qquad -1\leq k\leq n-1,
		 \end{align*}
	 where $\omega:= \sum_{p=1}^{n-1} y_{2p}\wedge y_{2p+1}$. If $n=1$, the space of invariant Killing spinors on $(M,g)$ has a basis given by $1,y_1$. 
	 \item The bundles $E_i$, $i=1,2,3$ are spanned over $C^{\infty}(M)$ by the Killing spinors $\Psi_{E_i,0},\Psi_{E_i,1}$ given as follows:
	 \begin{align*}
	 	\Psi_{E_1,0}&:= 1, \qquad \Psi_{E_1,1}:= y_1\wedge \omega^{n-1}, \qquad \Psi_{E_2,0} := \sum_{k=0}^{\lfloor \frac{n-1}{2} \rfloor} \frac{(-1)^k}{(2k+1)!} \psi_{2k},\\
	 	\Psi_{E_2,1}&:= \sum_{k=0}^{\lfloor \frac{n}{2} \rfloor} \frac{(-1)^k}{(2k)!} \psi_{2k-1},  \qquad \Psi_{E_3,0} := \sum_{k=0}^{\lfloor \frac{n-1}{2} \rfloor} \frac{1}{(2k+1)!} \psi_{2k}, \qquad  \Psi_{E_3,1}:= \sum_{k=0}^{\lfloor \frac{n}{2} \rfloor} \frac{1}{(2k)!} \psi_{2k-1}. 
	 \end{align*} 
	 \end{enumerate} 
\end{theorem}

\begin{table}[h!]
	\centering
	\begin{tabular}{ |l||l|l|l| }
		\hline
		$\dim(M)$  &  $\Psi_{E_1,0}$ & $\Psi_{E_2,0}$ & $\Psi_{E_3,0}$ \\
		\hline \hline 
		$7$   & $1$       &  $\omega-iy_1$   &  $\omega-iy_1$ \\
		$11$  & $1$     & $\omega-iy_1+\frac{1}{2}i y_1\wedge \omega^2$  & $\omega -iy_1 - \frac{1}{2}y_1\wedge \omega^2$ \\
		$15$  & $1$      & $\omega-iy_1 +\frac{1}{2}i y_1\wedge \omega^2 -\frac{1}{6}\omega^3$ & $\omega -iy_1 - \frac{1}{2}y_1\wedge \omega^2 +\frac{1}{6}\omega^2$ \\
		\hline\hline 
		$\dim(M)$ & $\Psi_{E_1,1}$ & $\Psi_{E_2,1}$ & $\Psi_{E_3,1}$ \\ \hline \hline 
		$7$   & $y_1\wedge \omega$       & $1+iy_1\wedge \omega$ & $1-iy_1\wedge \omega$ \\
		$11$  & $y_1\wedge \omega^2$      & $1 + iy_1\wedge \omega   -\frac{1}{2}\omega^2 $  &  $1 - iy_1\wedge \omega +\frac{1}{2}\omega^2 $ \\
		$15$  & $y_1\wedge \omega^3$      &  $1 +iy_1\wedge \omega -\frac{1}{2}\omega^2  - \frac{1}{6} iy_1\wedge \omega^3$ & $1 -iy_1\wedge \omega +\frac{1}{2}\omega^2  - \frac{1}{6} iy_1\wedge \omega^3$\\
		\hline
	\end{tabular}
	\caption{Bases for the $E_i$ in Low Dimensions \cite{Hof22}}
	\label{Tab:spinors_PsiEi_low_dim}
\end{table}

In low dimensions, explicit formulas for the spinors $\Psi_{E_i,0},\Psi_{E_i,1}$ spanning the $E_i$ can be found in \cite{Hof22}, and we recall them in Table \ref{Tab:spinors_PsiEi_low_dim}. The preceding theorem gives a detailed picture of the situation for $\ts$ homogeneous spaces, however until now little was known in the $\tad$ setting in dimension $>7$. Recalling that $\tad$ manifolds are almost never $\nabla^g$-Einstein (see \cite[Prop.\@ 2.3.3]{3str}), it follows that they do not generally admit Riemannian Killing spinors. In the subsequent two sections we study solutions of two differential equations generalizing the Killing equation. We examine first the equation that is obtained by deforming the $\ts$ Killing spinors with respect to a certain natural class of deformations on the underlying manifold (the so-called \emph{$\mathcal{H}$-homothetic deformations}). Secondly, we consider solutions of a different generalization of the Killing equation, whose solutions are sections of the (natural generalizations of the) bundles $E_i$ defined above. These have the advantage that they are defined for all $\tad$ manifolds, rather than only those which can be obtained from $\ts$ manifolds by $\mathcal{H}$-homothetic deformations (i.e.\@ the positive ones).

\section{$\mathcal{H}$-Homothetic Deformations of Killing Spinors on Positive $\tad$ Spaces}\label{section:actual_deformedKS}

In this section we examine how Killing spinors on $\ts$ manifolds transform under the \emph{$\mathcal{H}$-homothetic deformations} introduced in \cite[\S2.3]{3str}. We begin by recalling the definition of these deformations and their basic properties:
\begin{definition}(\cite[\S2.3]{3str}).\label{H_deformation_definition} Let $a,b,c\in \R$ be real numbers satisfying $a>0$, $a+b>0$, $c\neq 0$, $c^2=a+b$, and define $\alpha':= \alpha c/a$, $\delta' = \delta/c$. The $\mathcal{H}$-homothetic deformation of a $\tad$ manifold $(M,g,\xi_i,\eta_i,\varphi_i)$ relative to $a,b,c$ is the $3$-$(\alpha',\delta')$-Sasaki manifold $(M,g',\xi_i',\eta_i',\varphi_i')$ with structure tensors given by
	\[
	\eta_i' = c \eta_i , \qquad \xi_i' = \frac{1}{c}\xi_i, \qquad \varphi_i'=\varphi_i, \qquad g' = ag + b\sum_{i=1}^3 \eta_i\otimes \eta_i.   
	\]
\end{definition}
In particular, observe that $\mathcal{H}$-homothetic deformations satisfy $\alpha'\delta' = \alpha\delta/a$, so it is impossible to change type (positive, degenerate, negative) by such a deformation. It is demonstrated in \cite[\S2.3]{3str} that every positive ($\alpha\delta>0$) structure is an $\mathcal{H}$-homothetic deformation of a $3$-$\hat{\alpha}$-Sasakian structure ($\hat{\alpha}=\hat{\delta}$), every negative ($\alpha\delta<0$) structure is an $\mathcal{H}$-homothetic deformation of a negative structure with $\hat{\alpha} =- \hat{\delta}$, and any $\mathcal{H}$-homothetic deformation of a degenerate ($\delta=0$) structure is again degenerate.

Starting with a $\ts$ structure ($\alpha_0=\delta_0=1$), it is easy to see from the preceding definition that the parameters $a,b,c$ needed to obtain a $\tad$ structure with a given $\alpha,\delta$ satisfying $\alpha\delta >0$ are
\begin{align}
	\label{3Sas_def_parameters} a= \frac{1}{\alpha\delta} ,\qquad b  = \frac{\alpha-\delta}{\alpha\delta^2}, \qquad c = \frac{1}{\delta}. 
\end{align}
In particular, we have
\begin{align}\label{deformed_tensor_definitions}
\eta_i' = \frac{1}{\delta} \eta_i ,\qquad \xi_i'=\delta \xi_i ,\qquad \varphi_i'=\varphi_i, \qquad g' = \frac{1}{\alpha\delta}g + \frac{\alpha-\delta}{\alpha\delta^2} \sum_{i=1}^3\eta_i\otimes \eta_i. 
\end{align}
\begin{remark}
Since we are interested in deforming the $\ts$ Killing spinors along this family we need the orientation to remain consistent, so we shall only consider deformations resulting in $\tad$ structures with $\alpha,\delta>0$. 
\end{remark}
In what follows, we adapt the argument from \cite[Chap.\@ 2.2]{Julia_BB_PhD_thesis} to the $3$-contact setting in order to see how the Riemannian Killing spinors on the original Sasakian manifold change under the $\mathcal{H}$-homothetic deformation (\ref{3Sas_def_parameters}). As a matter of notation, we denote by $g$ the $\ts$ metric and by $g'$ the $\tad$ metric, and we consider the isometric identifications $\sigma \: (TM,g)\to (TM,g')$ of the tangent spaces given by
\begin{align}
	\label{small_j_def} \sigma(X) = \sqrt{\alpha\delta} X + (\delta-\sqrt{\alpha\delta})\sum_{p=1}^3 \eta_p(X) \xi_p \qquad \text{ for all } X \in TM.
\end{align}   
Concretely, $\sigma $ acts as multiplication by $\sqrt{\alpha\delta}$ on $\mathcal{H}$ and multiplication by $\delta$ on $\mathcal{V}$. A tedious but straightforward calculation shows that the Levi-Civita connection of the $\tad$ metric is simply the sum of the original Levi-Civita connection and a correction term depending on the original structure tensors:
\begin{lemma}\label{deformed_LC_formula}
The Levi-Civita connection $\nabla^{g'}$ of the $\tad$ structure obtained from a $\ts$ structure by the $\mathcal{H}$-homothetic deformation (\ref{3Sas_def_parameters}) is given by 
\[
\nabla^{g'}_XY = \nabla^{g}_XY + \frac{\delta-\alpha}{\delta} \sum_{p=1}^3 [\eta_p(X)\varphi_p(Y) + \eta_p(Y)\varphi_p(X)]. 
\]
\end{lemma}
Letting $\Sigma$ (resp.\@ $\widetilde{\Sigma}$) denote the spinor bundle of $(M,g)$ (resp.\@ $(M,g')$), there is an identification $ \Sigma \to\widetilde{\Sigma}$, $\psi\mapsto \widetilde{\psi}$ which relates the Clifford multiplication between the two spinor bundles via $\sigma $:
\begin{align}\label{tilde_relation_Cliff_mult} \widetilde{X\cdot \psi} = 	\sigma (X)\cdot \widetilde{\psi}  \qquad \text{ for all } \psi\in \Sigma, X\in TM.
\end{align}
In the following lemma we introduce another metric connection on $(TM,g')$ whose spinorial lift has a particularly useful property:
\begin{lemma}\label{nabla_sigma_definition}
	The connection $\nabla^{\sigma} $ on $(TM,g')$ defined by \[
	\nabla^{\sigma}_XY:= (\sigma \circ \nabla^{g}_X\circ \sigma^{-1})Y = \sigma (\nabla^{g}_X (\sigma^{-1}(Y))), \qquad  \text{ for all } X,Y\in TM
	\] is metric with respect to $g'$, and is given explicitly by
	\begin{align*}
	\nabla^{\sigma}_XY &= \nabla_X^{g}Y +\left( 1- \frac{\delta}{\sqrt{\alpha\delta}} \right) \sum_{p=1}^3 \Phi_p(X,Y)\xi_p + \left( 1- \frac{\sqrt{\alpha\delta}}{\delta} \right) \sum_{p=1}^3 \eta_p(Y) \varphi_p(X)^{\mathcal{H}} \\
	& \qquad  - \left(   1-\frac{\sqrt{\delta}}{\sqrt{\alpha}} \right)\sum_{p=1}^3 \eta_p(Y) \varphi_p(X)^{\mathcal{V}},
	\end{align*}
where superscript $\mathcal{H}$ and $\mathcal{V}$ denote projection onto the horizontal and vertical spaces respectively. Furthermore, its spinorial lift satisfies \begin{align}\label{j_derivative_relation} \nabla_X^{\sigma} \widetilde{\psi} = \widetilde{ \nabla_X^{g} \psi},\qquad \text{ for all } \psi\in \Sigma, X\in TM.
\end{align}
\end{lemma}
\begin{proof}
	We directly calculate:
	\begin{align*}
	&	\nabla_X^{\sigma} Y = \sigma\left(\nabla_X^{g} \left[ \frac{1}{ \sqrt{\alpha\delta}} Y + \left(\frac{1}{\delta} - \frac{1}{\sqrt{\alpha\delta}} \right)  \sum_{p=1}^3 \eta_p(Y)\xi_p    \right]   \right) \\
	=& \frac{1}{\sqrt{\alpha\delta}} \sigma ( \nabla_X^{g} Y) +\left(\frac{1}{\delta} - \frac{1}{\sqrt{\alpha\delta}} \right) \sum_{p=1}^3 \sigma (\nabla_X^{g}(\eta_p(Y)\xi_p)) \\
		=& \frac{1}{\sqrt{\alpha\delta}} \left[ \sqrt{\alpha\delta}\ \nabla_X^{g} Y + (\delta-\sqrt{\alpha\delta}) \sum_{p=1}^3\eta_p(\nabla_X^{g}Y)\xi_p \right]    +\left(\frac{1}{\delta} - \frac{1}{\sqrt{\alpha\delta}} \right) \sum_{p=1}^3 \sigma ( X(\eta_p(Y)) \xi_p  + \eta_p(Y) \nabla_X^{g} \xi_p     ),
		\end{align*}
		and hence the difference of $\nabla^{\sigma}$ and $\nabla^g$ is given by
		\begin{align*}
		&\nabla_X^{\sigma}Y - \nabla^g_XY \\
		=&   \left(\frac{\delta}{\sqrt{\alpha\delta}} -1\right) \sum_{p=1}^3 \left[   X(\eta_p( Y)) - (\nabla_X^{g}\eta_p)Y \right]\xi_p  + \left(\frac{1}{\delta} - \frac{1}{\sqrt{\alpha\delta}} \right) \sum_{p=1}^3 \left[ \delta X(\eta_p(Y)) \xi_p -\eta_p(Y)\sigma (\varphi_p(X))  \right]\\
		=&  -  \left( \frac{\delta}{\sqrt{\alpha\delta}} -1 \right)  \sum_{p=1}^3  \Phi_p(X,Y) \xi_p - \left( \frac{1}{\delta}-\frac{1}{\sqrt{\alpha\delta}}\right) \sum_{p=1}^3 \eta_p(Y)\left[    \sqrt{\alpha\delta}\ \varphi_p(X)^{\mathcal{H}}  +\delta\varphi_p(X)^{\mathcal{V}}   \right] \\
		=& \left( 1- \frac{\delta}{\sqrt{\alpha\delta}} \right) \sum_{p=1}^3 \Phi_p(X,Y)\xi_p + \left( 1- \frac{\sqrt{\alpha\delta}}{\delta} \right) \sum_{p=1}^3 \eta_p(Y) \varphi_p(X)^{\mathcal{H}}  - \left(   1-\frac{\sqrt{\delta}}{\sqrt{\alpha}} \right)\sum_{p=1}^3 \eta_p(Y) \varphi_p(X)^{\mathcal{V}}.
	\end{align*}
The fact that $\nabla^{\sigma}$ is metric with respect to $g'$ is easily seen by noting that
\begin{align*} 
g'(\nabla^{\sigma}_XY,Z) + g'(Y,\nabla_X^{\sigma} Z) &= g(\nabla_X^{g}(\sigma^{-1}Y),\sigma^{-1}Z) + g(\sigma^{-1}Y,\nabla_X^{g}(\sigma^{-1}Z) ) \\ 
&= Xg(\sigma^{-1}Y,\sigma^{-1}Z) =Xg'(Y,Z). 
\end{align*}
Finally, (\ref{j_derivative_relation}) follows essentially by definition of the spin lift of a metric connection, using (\ref{tilde_relation_Cliff_mult}). 
\end{proof}
The reason we consider $\nabla^{\sigma}$ is that it serves as an intermediate step in determining how $\nabla^{g'}$ acts on Killing spinors in $\widetilde{\Sigma}$. Indeed, since $\nabla^{\sigma}$ is a metric connection we have $\nabla^{g'} = \nabla^{\sigma} + \tau$ for some difference tensor $\tau \in TM^* \otimes \Lambda^2 TM$, and the right hand side of (\ref{j_derivative_relation}) has a very simple form when $\psi$ is a Killing spinor on the $\ts$ manifold $(M,g)$. Thus, in order to differentiate such a deformed Killing spinor $\widetilde{\psi}$ all that remains is to describe its Clifford product by $\tau_V$ for all $ V\in TM$, as we do in the following lemma.
\begin{lemma}\label{difference_tensor_lemma}
	The difference tensor $\tau := \nabla^{g'} - \nabla^{\sigma} \in TM^* \otimes \Lambda^2TM$ satisfies
	\[
	\tau_X = (\alpha-\sqrt{\alpha\delta}) \sum_{p=1}^3 \eta_p' \wedge (X\lrcorner\Phi_p') , \qquad \tau_{\xi_i'} = (\alpha-\delta) (\Phi_i')^{\mathcal{H}} 
	\]
	for all $X\in \mathcal{H}$ and all $Y,Z\in TM$.  
\end{lemma}
\begin{proof}
	Let $Y,Z\in TM$ be arbitrary. For $X\in \mathcal{H}$ we compute, using the preceding two lemmas,
	\begin{align*}
		&\tau_X(Y,Z) = g'(\nabla^{g'}_X Y - \nabla^{\sigma}_XY, Z) \\
		&=\sum_{p=1}^3  g'\left(  \frac{\delta-\alpha}{\delta}  \eta_p(Y)\varphi_p(X) - \left( 1- \frac{\delta}{\sqrt{\alpha\delta}} \right)  \Phi_p(X,Y)\xi_p - \left( 1- \frac{\sqrt{\alpha\delta}}{\delta} \right)  \eta_p(Y) \varphi_p(X) ,Z   \right) \\
		&= \frac{1}{\alpha\delta}\sum_{p=1}^3\left[   \frac{\alpha-\sqrt{\alpha\delta} }{\delta}\eta_p(Y)\Phi_p(X,Z) -\left( 1-\frac{\delta}{\sqrt{\alpha\delta}}\right) \eta_p(Z)\Phi_p(X,Y)   \right] \\
		& \qquad  -  \frac{\alpha-\delta}{\alpha\delta^2}\left(1-\frac{\delta}{\sqrt{\alpha\delta}}\right) \sum_{p=1}^3  \eta_p(Z) \Phi_p(X,Y)  \\
		&= \frac{\alpha-\sqrt{\alpha\delta}}{\alpha\delta^2} \sum_{p=1}^3 \left[  \eta_p(Y)\Phi_p(X,Z)  - \eta_p(Z)\Phi_p(X,Y)     \right] = \frac{\alpha-\sqrt{\alpha\delta}}{\alpha\delta^2} \sum_{p=1}^3 \left( \eta_p\wedge (X\lrcorner \Phi_p)\right)(Y,Z)\\
		&= ( \alpha-\sqrt{\alpha\delta}) \sum_{p=1}^3 (\eta_p' \wedge (X\lrcorner\Phi_p')) (Y,Z). 
	\end{align*}
In the vertical directions the difference between $\nabla^{g'}$ and $\nabla^{\sigma}$ is given by
\begin{align*}
	&\nabla^{g'}_{\xi_i'}Y - \nabla^{\sigma}_{\xi_i'}Y  = \delta \left( \nabla^{g'}_{\xi_i}Y - \nabla^{\sigma}_{\xi_i}Y \right)  \\
	=&    (\delta-\alpha)(\varphi_i(Y) + \eta_k(Y) \xi_j -\eta_j(Y)\xi_k  )  - \left( \delta -\frac{\delta^2}{\sqrt{\alpha\delta}} \right) \sum_{p=1}^3 \Phi_p(\xi_i,Y)\xi_p  +\left(\delta -\frac{\delta\sqrt{\delta}}{\sqrt{\alpha}} \right) \sum_{p=1}^3 \eta_p(Y) \varphi_p(\xi_i)  
	\end{align*}
for all $Y \in TM$. Simplifying, using the formula (\ref{phii_vertical_directions}) describing the action of the endomorphisms $\varphi_p$ (and the associated $2$-forms $\Phi_p$) in the vertical directions and the formula (\ref{deformed_tensor_definitions}) for $g'$, we obtain
\begin{align*}
	& \tau_{\xi_i'}(Y,Z) = \frac{1}{\alpha} \left[ \frac{\delta-\alpha}{\delta}\left(  -\Phi_i(Y,Z) -\eta_{j,k}(Y,Z) \right) + \left(1-\frac{\delta}{\sqrt{\alpha\delta}}\right) \eta_{j,k}(Y,Z) - \left( 1-\frac{\sqrt{\delta}}{\sqrt{\alpha}} \right)  \eta_{j,k}(Y,Z)  \right] \\
	& \qquad \qquad +\frac{\alpha-\delta}{\alpha\delta}  \left[ \left(1- \frac{\delta}{\sqrt{\alpha\delta}}\right) \eta_{j,k}(Y,Z) - \left( 1-\frac{\sqrt{\delta}}{\sqrt{\alpha}} \right) \eta_{j,k}(Y,Z)  \right]  \\
	=& \frac{\alpha-\delta}{\alpha\delta} \left( \Phi_i + \eta_{j,k} \right)(Y,Z) = (\alpha-\delta) (\Phi_i')^{\mathcal{H}}(Y,Z),             
\end{align*}
where $(i,j,k)$ is an even permutation of $(1,2,3)$.
\end{proof}
As a consequence we obtain the following theorem specifying the differential equations of the deformed Killing spinors on $\tad$ manifolds with respect to the Levi-Civita connection. Beginning with this theorem and continuing for the remainder of the article, we shall always use $(g,\xi_i,\eta_i,\varphi_i)$ to denote a $\tad$ structure rather than a $\ts$ structure.
\begin{theorem}\label{deformed_KS_theorem_ref}
	A positive $\tad$ manifold $(M^{4n-1},g,\xi_i,\eta_i,\varphi_i)$ of dimension $4n-1$ admits at least $n+1$ linearly independent spinors satisfying
	\begin{align}\label{KS_with_torsion}
		\nabla_{\xi_i}^{g}\psi  &=  \frac{\delta }{2} \xi_i\cdot \psi  + \frac{\alpha-\delta}{2} \Phi_i^{\mathcal{H}} \cdot\psi  , \qquad \nabla^{g}_X \psi = \frac{\sqrt{\alpha\delta} }{2} X\cdot \psi  + \frac{\sqrt{\alpha\delta}-\alpha }{2}\sum_{p=1}^3\xi_p\cdot  \varphi_p(X) \cdot \psi  
	\end{align}
for all $X \in \mathcal{H}$. If $(M,g)$ is isometric to the round sphere $S^{4n-1}$ (with its standard $\ts$ structure), then the spinor bundle is spanned by solutions of (\ref{KS_with_torsion}). 
\end{theorem}
\begin{proof}
	The formulas follow from a straightforward calculation using (\ref{j_derivative_relation}) and Lemma \ref{difference_tensor_lemma}. The lower bound on the number of solutions is a consequence of the well-known fact that a $\ts$ manifold of dimension $4n-1$ admits $n+1$ linearly independent Killing spinors if it is not isometric to the round sphere $S^{4n-1}$ (with its standard $\ts$ structure) and a full basis of Killing spinors otherwise (see e.g.\@ \cite{Bar}).
\end{proof}

We introduce a name for spinors satisfying (\ref{KS_with_torsion}). 
\begin{definition}
	A \emph{deformed Killing spinor} on a positive $\tad$ manifold is a solution of (\ref{KS_with_torsion}). 
\end{definition}

\begin{remark}
	Deformed Killing spinors, as per the preceding definition, are not to be confused with the solutions of (\ref{deformedKillingspinorsbundle}) in the next section, which we shall call \emph{$\mathcal{H}$-Killing spinors}. This is an unfortunate clash of terminology with the previous work \cite{AHLspheres}, which used the name deformed Killing spinors to describe solutions of (\ref{deformedKillingspinorsbundle}), however we believe that spinors obtained from Killing spinors by $\mathcal{H}$-homothetic deformations are much more deserving of the name "deformed". 
\end{remark}

\begin{remark}
Conceptually, the existence of spinors satisfying (\ref{KS_with_torsion}) emphasizes an important point: $\ts$ structures exist within the large family of $\tad$ structures, and the spinors on these spaces should also vary smoothly along the family. Equation (\ref{KS_with_torsion}) says that these spinors are Killing spinors up to a correction term, defined in terms of the structure tensors, which vanishes in the $3$-$\alpha$-Sasakian case ($\alpha=\delta$).
\end{remark}

In the following proposition we collect the spinorial equations satisfied by solutions of (\ref{KS_with_torsion}) with respect to the other relevant spinorial connections, namely the canonical connection $\nabla$ of the $\tad$ structure and the connection $\nabla^{\sigma}$ defined in Lemma \ref{nabla_sigma_definition}.

\begin{proposition} Deformed Killing spinors are generalized Killing spinors with torsion for the metric connection $\nabla^{\sigma}$, satisfying the equation
\[  \nabla_X^{\sigma} \psi = \frac{1}{2}\sigma(X)\cdot \psi = \left(\frac{\delta }{2} X^{\mathcal{V}} + \frac{\sqrt{\alpha\delta}}{2}X^{\mathcal{H}} \right)\cdot \psi  \] 
for all $X \in TM$. With respect to the canonical connection, $\nabla$, they satisfy
\begin{align}
\nabla_{\xi_i} \psi & = \frac{\delta}{2}\xi_i\cdot \psi +  \frac{2\alpha-\delta }{2} \Phi_i^{\mathcal{H}} \cdot \psi + \frac{\delta-4\alpha}{2} \eta_j\cdot \eta_k\cdot \psi , \label{canonical_deformed_vertical_direction} \\
\nabla_X \psi&= \frac{\sqrt{\alpha\delta}}{2}X\cdot \psi + \frac{ \sqrt{\alpha\delta}}{2} \sum_{p=1}^3 \xi_p\cdot \varphi_p(X)\cdot \psi  \label{canonical_deformed_horizontal_direction}
\end{align}
for all $X\in \mathcal{H}$.
\end{proposition}
\begin{proof}
	The differential equation with respect to $\nabla^{\sigma}$ follows immediately from (\ref{j_derivative_relation}). The other equations follow from an easy calculation using the formula (\ref{canonical_torsion_formula}) for the torsion of $\nabla$.
\end{proof}

\begin{remark} The preceding proposition lends credence to the notion that $\nabla^{\sigma}$ is better suited to the deformed Killing spinors than either the Levi-Civita or canonical connections of the $\tad$ structure. This is not altogether surprising, given that the Levi-Civita connection on a $\ts$ manifold has good spinorial properties, and $\nabla^{\sigma}$ is its natural extension under $\mathcal{H}$-homothetic deformations.
\end{remark}

In practice, although deformed Killing spinors behave nicely with respect to $\nabla^{\sigma}$, they also have several unfortunate properties:
\begin{enumerate}[(i)]
	\item They are not typically Killing spinors for the Levi-Civita connection in the horizontal or vertical directions;
	\item They do not generalize well to negative $\tad$ manifolds due to the appearance of coefficients involving $\sqrt{\alpha\delta}$.
\end{enumerate}
Let us elaborate on point (ii). At present, the main techniques for finding globally defined spinor fields are quite limited: one can look for bundles of spinors preserved by a connection whose restriction to the bundle is flat (as Friedrich and Kath did in \cite{Fried90}), or one can use the existence of a $\text{G}$-structure to find trivial subbundles for the (lifted) action of $G$ on the spinor module (important examples of this include invariant spinors on homogeneous spaces and the defining spinors of $\SU(3)$- and $\G_2$-structures in dimensions $6$ and $7$ respectively). The first technique is ruled out for our purposes since the natural choice of connection to consider would be $\bar{\nabla}$ defined by \[\bar{\nabla}_{\xi_i}\psi  = \nabla^{g}_{\xi_i}\psi   -   \frac{\delta}{2}\xi_i\cdot \psi - \frac{\alpha-\delta}{2} \Phi_i^{\mathcal{H}} \cdot \psi  ,\quad \bar{\nabla}_X \psi = \nabla^{g'}_X \psi - \frac{\sqrt{\alpha\delta} }{2} X\cdot \psi -\frac{\sqrt{\alpha\delta} - \alpha }{2} \sum_{p=1}^3\xi_p\cdot \varphi_p(X) \cdot \psi      \]
for $X\in \mathcal{H}$, and this doesn't generalize well to the negative $\tad$ setting in light of point (ii) above; one would need to consider spinorial equations with complex coefficients, which deviate from the spirit of classical real and imaginary Killing spinors. The second technique is also not straightforward in our situation; while there is a reduction to $\Sp(n-1)\cong \Sp(n-1)\times \{1\}\subset \SO(4n-1)$ (\cite{3Sas_structure_reduction}) and this group is known to stabilize $2n$ linearly independent spinors (see e.g.\@ \cite{AHLspheres}), the differential equations satisfied by these spinors depend on the intrinsic torsion of the $\Sp(n-1)$-structure, which is not well understood at present. The rest of the paper is devoted to a different generalization of Killing spinors, introduced in Theorem \ref{deformedKillingspinorstheorem}, which avoids these problems. We will show that these spinors can be defined for any $\tad$ manifold, including when $\alpha\delta<0$, and that they behave well in the homogeneous setting with respect to a certain notion of duality which we introduce in Section \ref{section:duality}.

\section{$\mathcal{H}$-Killing Spinors on $\tad$ Spaces}\label{section:horizontalKS}
Let $(M,g,\varphi_i,\xi_i,\eta_i)$ be a $\tad$ manifold with Levi-Civita connection $\nabla^g$ and canonical connection $\nabla$. In this section we will prove that the bundles $E_i$ recalled in (\ref{Ei_definition}) are spanned by spinors satisfying the \emph{$\mathcal{H}$-Killing equation} (\ref{deformedKillingspinorsbundle}). To construct such spinors we adapt the method of Friedrich and Kath in \cite{Fried90}: we prove that the bundles $E_i$ are preserved by a certain modified connection $\widehat{\nabla}$, and that the restriction of $\widehat{\nabla}$ to $E:= E_1+E_2+E_3$ is flat. The connection $\widehat{\nabla}$ will be chosen in an obvious way so that its parallel sections are precisely $\mathcal{H}$-Killing spinors, and we shall assume that the underlying manifold is simply-connected so that flatness of $\widehat{\nabla}\rvert_E$ implies the existence of globally defined parallel sections. While the overall strategy of Friedrich and Kath still works in this new setting, the calculations become considerably more complicated than in the $\ts$ case. We will mitigate this by using the canonical connection--which was only recently introduced in \cite{3str}--where possible, in order to take advantage of the curvature identities calculated in \cite{ADScurv}. The main result of this section is the following:

\begin{theorem}\label{deformedKillingspinorstheorem} On a simply-connected $\tad$ space $(M,g, \varphi_i,\xi_i,\eta_i)$, the bundle $E:=E_1+E_2+E_3$ has rank at least $2$ and is spanned by solutions of the equation
\begin{align} \label{deformedKillingspinorsbundle} 
\nabla^g_X\psi &=  \frac{\alpha}{2} X\cdot \psi +\frac{\alpha-\delta}{2}\sum_{p=1}^3\eta_p(X) \Phi_p\cdot \psi \qquad \text{ for all } X\in TM.
\end{align}
\end{theorem}

\begin{definition}
	We call solutions of (\ref{deformedKillingspinorsbundle}) \emph{$\mathcal{H}$-Killing spinors}.
\end{definition}	
\begin{remark}
	Note that $\mathcal{H}$-Killing spinors are generalizations of Killing spinors (which may be obtained by setting $\alpha=\delta$), and also avoid several of the problems encountered with deformed Killing spinors in the previous section: they are Killing spinors in the horizontal directions, and are well-defined for all $\tad$ manifolds (not just the positive ones).  
\end{remark}

The rest of the section is devoted to the proof of Theorem \ref{deformedKillingspinorstheorem}.
\subsection{The Modified Spinorial Connection}

In this subsection we introduce a \emph{modified connection} $\widehat{\nabla}$ on the spinor bundle and show that it preserves the bundles $E_i$, $i=1,2,3$. To begin, we prove two technical lemmas which will allow us to adapt the calculation of Friedrich and Kath in \cite{Fried90} to this new setting. The first lemma describes how sections of $E_i$ behave with respect to Clifford multiplication by the structural $2$-forms $\Phi_j,\Phi_k$, which is necessary to understand due to the second term on the right hand side of the $\mathcal{H}$-Killing equation (\ref{deformedKillingspinorsbundle}).
\begin{lemma}
	Let $(M,g,\xi_i,\eta_i,\varphi_i)$ be a $\tad$ manifold. If $\psi\in \Gamma(E_i)$ and $(i,j,k)$ is an even permutation of $(1,2,3)$, then
	\begin{align}
		(-2\varphi_i(X)+\xi_i X - X\xi_i )\cdot \Phi_j\cdot \psi &= [8\varphi_k(X) -2\xi_k X+2X \xi_k  -4\eta_i(X)\xi_j +4\eta_j(X)\xi_i ]\cdot \psi , \label{degthreeproduct1}\\
		(-2\varphi_i(X)+\xi_i X-X\xi_i)\cdot \Phi_k\cdot \psi &=  [-8\varphi_j(X)+2\xi_j X-2X\xi_j -4\eta_i(X)\xi_k+4\eta_k(X)\xi_i]\cdot \psi ,  \label{degthreeproduct2}
	\end{align}
	for all $X\in\vect(M)$.
\end{lemma}
\begin{proof}
	We calculate in the Clifford algebra,
	\begin{align*}
		-2\varphi_i(X)\cdot \Phi_j &= 2(2\varphi_i(X)\lrcorner \Phi_j - \Phi_j\cdot \varphi_i(X)) = 2( -2\varphi_j(\varphi_i(X)) - \Phi_j\cdot \varphi_i(X) )\\
		&= -4\eta_i(X)\xi_j +4\varphi_k(X) -2\Phi_j\cdot \varphi_i(X) , \\
		\xi_i\cdot X \cdot \Phi_j &= \xi_i\cdot (-2X\lrcorner \Phi_j + \Phi_j\cdot X)= 2\xi_i\cdot \varphi_j(X) +\xi_i \cdot \Phi_j \cdot X \\
		&= 2\xi_i\cdot \varphi_j(X)+( -2\xi_i\lrcorner \Phi_j  +\Phi_j\cdot \xi_i)\cdot X\\
		&=2\xi_i\cdot \varphi_j(X) -2\xi_k\cdot X +\Phi_j\cdot \xi_i\cdot X, \\
		-X\cdot \xi_i\cdot \Phi_j &= -X\cdot ( -2\xi_i\lrcorner \Phi_j + \Phi_j\cdot \xi_i) = 2X\cdot\xi_k +(2X\lrcorner \Phi_j - \Phi_j\cdot X)\cdot \xi_i \\ &= 2X\cdot \xi_k -2\varphi_j(X) \cdot \xi_i  -\Phi_j\cdot X\cdot \xi_i .
	\end{align*}
	Adding these equations and using the defining relation for $E_i$ gives
	\begin{align*}
		&(-2\varphi_i(X)+\xi_i X-X \xi_i)\cdot \Phi_j\cdot \psi \\
		=& [4\varphi_k(X) -2\xi_k X+2X\xi_k]\cdot \psi +[-4\eta_i(X)\xi_j + 2\xi_i\cdot \varphi_j(X) -2 \varphi_j(X)\cdot \xi_i]\cdot \psi \\
	=&  [4\varphi_k(X) -2\xi_k X+2X\xi_k]\cdot \psi +[-4\eta_i(X)\xi_j + 4\varphi_i(\varphi_j(X))]\cdot \psi \\
		=&  [4\varphi_k(X) -2\xi_k X+2X\xi_k]\cdot \psi +[-4\eta_i(X)\xi_j + 4\varphi_k(X) +4\eta_j(X)\xi_i ]\cdot \psi \\
		=& [8\varphi_k(X) -2\xi_kX+2X\xi_k -4\eta_i(X) \xi_j +4\eta_j(X)\xi_i]\cdot \psi.
	\end{align*}
	Equation (\ref{degthreeproduct2}) is proved analogously.
\end{proof}
The second technical lemma is an identity which arises from Friedrich and Kath's calculations for the $\ts$ case, and also holds for $\tad$ manifolds:
\begin{lemma}\emph{(Based on the proof of \cite[Thm.\@ 1]{Fried90}).} If $(M,g,\xi_i,\eta_i,\varphi_i)$ is a $\tad$ manifold and $\psi\in \Gamma(E_i)$, then
	\begin{align}\label{technicallemma2} 
	 [-2 (g(X,Y)\xi_i -\eta_i(X)Y)- \varphi_i(Y) X +  X\varphi_i(Y) ] \cdot \psi +(-2\varphi_i(X) +\xi_i X-X\xi_i)\cdot\left( \frac{1}{2} Y\cdot \psi    \right) =0.
	\end{align}
\end{lemma}
\begin{proof}
	This follows by the same calculation as on \cite[p.\@ 547]{Fried90} (note that their calculation has a small typo on the second last line; the term "$+XY\xi$" should instead say "$\pm XY\xi$").
\end{proof}

Using these two lemmas we prove the desired result:
\begin{proposition}
	If $(M,g,\xi_i,\eta_i,\varphi_i)$ is a $\tad$ manifold then the modified spinorial connection
	\begin{align}\label{modified_connection_formula}
	\widehat{\nabla}_Y\psi  :=\nabla_Y^g \psi - \frac{\alpha}{2} Y\cdot \psi -  \frac{\alpha-\delta}{2} \sum_{p=1}^3 \eta_p(Y)\Phi_p \cdot \psi , \qquad Y\in TM, \ \psi \in \Sigma M
	\end{align} preserves the bundles $E_i$, $i=1,2,3$.
\end{proposition}
\begin{proof}
	Let $\psi\in \Gamma(E_i)$, and take $(i,j,k)$ an even permutation of $(1,2,3)$. Differentiating the defining equation $$(-2\varphi_i(X) +\xi_i X-X\xi_i)\psi =0  $$ with respect to $Y$ gives
	\begin{align}
		0&= [-2(\nabla^g_Y\varphi_i)X -2 \varphi_i(\nabla_Y^gX) +(\nabla_Y^g \xi_i)X + \xi_i(\nabla_Y^gX) - (\nabla_Y^gX)\xi_i - X(\nabla_Y^g \xi_i)]\cdot \psi \nonumber \\
		&\qquad \qquad +(-2\varphi_i(X) +\xi_i X-X\xi_i)\cdot \nabla_Y^g\psi\nonumber \\
		&= [-2(\nabla^g_Y\varphi_i)X  +(\nabla_Y^g \xi_i)X   - X(\nabla_Y^g \xi_i)]\cdot \psi \label{mastereq}\\
		&\qquad \qquad +(-2\varphi_i(X) +\xi_i X-X\xi_i)\cdot\left( \frac{\alpha}{2} Y\cdot \psi + \frac{\alpha-\delta}{2} \sum_{p=1}^3 \eta_p(Y)\Phi_p\cdot \psi    \right)\nonumber \\
		&\qquad \qquad  +(-2\varphi_i(X) +\xi_i X-X\xi_i)\cdot \widehat{\nabla}_Y\psi . \nonumber 
	\end{align}
	\begin{enumerate}[(i)]
		\item If $Y\in\mathcal{H}$ then (\ref{mastereq}) simplifies, using Proposition \ref{LCder}, to
			\begin{align*}
				0&= [-2\alpha (g(X,Y)\xi_i -\eta_i(X)Y)-\alpha \varphi_i(Y) X + \alpha X\varphi_i(Y) ] \cdot \psi \\
				& \qquad +(-2\varphi_i(X) +\xi_i X-X\xi_i)\cdot\left( \frac{\alpha}{2} Y\cdot \psi    \right)  +(-2\varphi_i(X) +\xi_i X-X\xi_i)\cdot \widehat{\nabla}_Y\psi,
			\end{align*}
		and the fact that the first two terms on the right hand side of the above equation sum to zero follows immediately from (\ref{technicallemma2}).
		\item If $Y=\xi_i $ then Proposition \ref{LCder} reduces (\ref{mastereq}) to 
		\begin{align*}
			0&=(-2\varphi_i(X) +\xi_i X-X\xi_i)\cdot\left(   \frac{\alpha}{2} \xi_i\cdot \psi + \frac{\alpha-\delta}{2} \Phi_i\cdot \psi    \right) +(-2\varphi_i(X) +\xi_i X-X\xi_i)\cdot \widehat{\nabla}_{\xi_i}\psi ,
		\end{align*}
		whose first term vanishes due to the fact that $\Phi_i$ acts on $E_i^-$ as a multiple of $\xi_i$, and $\xi_i$ anti-commutes with the operator defining $E_i^-$.
		\item If $Y=\xi_j$ then (\ref{mastereq}) reduces to
			\begin{align*}
				0&= -2[\alpha( \eta_j(X) \xi_i - \eta_i(X)\xi_j) +2(\alpha-\delta)\varphi_k(X)+(\alpha-\delta)\eta_j(X) \xi_i -(\alpha-\delta)\eta_i(X)\xi_j ]\cdot \psi \\
				&\qquad \qquad +[(-\alpha\xi_k +(\alpha-\delta)\xi_k)    X-  X (-\alpha\xi_k +(\alpha-\delta)\xi_k)   ]\cdot \psi \\
				&\qquad \qquad + (-2\varphi_i(X) +\xi_i X - X\xi_i )\cdot \left(\frac{\alpha}{2} \xi_j\cdot \psi+\frac{\alpha-\delta}{2} \Phi_j\cdot \psi\right)\\
				&\qquad \qquad +(-2\varphi_i(X)+\xi_i X -X \xi_i) \cdot \widehat{\nabla}_{\xi_j}\psi\\
				&=(\alpha-\delta)[-4\varphi_k(X) -2\eta_j(X)\xi_i +2\eta_i(X)\xi_j+\xi_k X-X\xi_k +\frac{1}{2} (-2\varphi_i(X)+\xi_iX- X\xi_i) \cdot \Phi_j] \cdot \psi\\
				&\qquad \qquad +(-2\varphi_i(X)+\xi_i X -X \xi_i) \cdot \widehat{\nabla}_{\xi_j}\psi
			\end{align*}
		(for the second equality we have again used (\ref{technicallemma2}) to eliminate some of the terms). The vanishing of the first term then follows immediately by substituting (\ref{degthreeproduct1}). 
		\item For the case $Y=\xi_k$, one performs an analogous calculation using (\ref{degthreeproduct2}).
	\end{enumerate}
\end{proof}

\subsection{Curvature and Torsion Identities for the Canonical Connection}

In this subsection we derive some useful curvature identities for the canonical connection, and also some torsion identities which will allow us to translate the former into curvature identities for the Levi-Civita connection. These will be used in the next subsection towards our ultimate goal of showing that the restriction of $\widehat{\nabla}$ to $E$ is flat. We denote by $\nabla^g$, $\nabla$, $\widehat{\nabla}$ the Levi-Civita, canonical, and modified connections, and by $R^g$, $R$, $\widehat{R}$ their respective curvature operators, and we recall from Proposition \ref{prelims:canonical_connection} the constant $\beta:=2(\delta-2\alpha)$. For the purposes of our calculations, it will be convenient to use a special type of orthonormal frame called an \emph{adapted frame}:
\begin{definition}
	An orthonormal frame $e_1,\dots, e_{4n-1}$ with $\xi_i = e_{i}$ for $i=1,2,3$ is said to be adapted if the fundamental 2-forms $\Phi_i$, $i=1,2,3$ take the form
\begin{align}
\Phi_1&= -\xi_{2,3} -\sum_{p=0}^{n-1} (e_{4p,4p+1} + e_{4p+2,4p+3}) ,  \label{Phi1}\\
\Phi_2&= \xi_{1,3} -\sum_{p=0}^{n-1}(e_{4p,4p+2} - e_{4p+1,4p+3}) ,  \label{Phi2} \\
\Phi_3&= -\xi_{1,2} -\sum_{p=0}^{n-1}(e_{4p,4p+3}+ e_{4p+1,4p+2}) . \label{Phi3} 
\end{align}
\end{definition}
From this point forward, we shall always work with adapted orthonormal frames. We begin by proving a curvature identity for a certain $\varphi_i$-twisted trace of the canonical curvature. The reason to consider such a twisting is the same as in the classical argument by Friedrich and Kath; if $\psi \in \Gamma(E_i)$ then the $E_i$-projection of the Clifford product of a $2$-form $\tau = \sum_{p<q} \tau_{p,q}e_p\wedge e_q $ with $\psi$ is given precisely by considering only the Clifford products with the terms $\tau_{p,q}e_p\wedge e_q$ where $e_q=\pm \varphi_i e_p$ (see the proof of \cite[Thm.\@ 1]{Fried90}). In particular, in the next subsection we will apply this to the curvature $2$-forms $R^g(X,Y)\in \Lambda^2TM$ associated to pairs of tangent vectors $X,Y\in TM$ in the course of proving that the restriction of $\widehat{\nabla}$ to $E$ is flat. To calculate these curvature $2$-forms we shall first calculate the corresponding $2$-forms for the canonical connection (taking advantage of its nice behaviour with respect to the $\tad$ structure tensors), and then translate these back in terms of the Levi-Civita connection using (\ref{curvdifference}). As such, we begin with a proposition giving identities for the $\varphi_i$-twisted traces of the canonical curvature operator $R$. The subsequent proposition gives the $\varphi_i$-twisted traces of certain operators defined in terms of the canonical torsion $T$ and its exterior derivative $dT$; this is precisely the information needed in order to translate the canonical curvature identities into identities for the Levi-Civita curvature using (\ref{curvdifference}).
\begin{proposition} \label{curvprop}
	If $(M,g,\xi_i,\eta_i,\varphi_i)$ is a $\tad$ manifold and $e_1,\dots,e_{4n-1}$ an adapted orthonormal (local) frame, then
	\[
	\sum_{s=1}^{4n-1} R(X,Y,e_s,\varphi_i(e_s))= 
	\begin{cases}
		4n\alpha\beta  \Phi_i(X,Y) , & X,Y\in \mathcal{H}, \\
		0,& X\in \mathcal{H},Y\in\mathcal{V},\\
		0, & X\in \mathcal{V},Y\in\mathcal{H},\\
		8n\alpha\beta  \Phi_i(X,Y) , & X,Y\in\mathcal{V}.
	\end{cases}
	\]
\end{proposition}
\begin{proof} We consider the cases one at a time:
	\begin{enumerate}[(i)]
		\item Suppose first that $X,Y\in\mathcal{H}$. For $e_s\in \mathcal{H}$ and any even permutation $(i,j,k)$ of $(1,2,3)$, it follows from \cite[Eqn.\@ (2.6)]{ADScurv} that
		\begin{align*}
			2\alpha\beta \Phi_i(X,Y) &= R(X,Y,e_s,\varphi_i(e_s)) + R(X,Y,\varphi_j(e_s),\varphi_k(e_s)) \\
			&= R(X,Y,e_s,\varphi_i(e_s)) + R(X,Y,\varphi_j(e_s),\varphi_i(\varphi_j(e_s))), 
		\end{align*}
		and taking the sum over the horizontal basis vectors $e_s$, $s=4,\dots,4n-1$ then gives
		\begin{align}\label{canonicalcurv_identity_a}
			2\sum_{s=4}^{4n-1} R(X,Y,e_s,\varphi_i(e_s))&= 2(4n-4)\alpha\beta \Phi_i(X,Y)
		\end{align}
		(since $\varphi_j(e_s)$ runs through the list $\pm e_4,\dots, \pm e_{4n-1}$ as $s$ runs through $4,\dots, 4n-1$). In the vertical directions, it follows from \cite[Eqns.\@ (1.2), (2.5)]{ADScurv} that $R(X,Y,\xi_j,\xi_k) = 2\alpha\beta \Phi_i(X,Y)$,
		and hence
		\begin{align}\label{canonicalcurv_identity_b}
			\sum_{s=1}^3 R(X,Y,e_s,\varphi_i(e_s)) &= 4\alpha\beta\Phi_i(X,Y).
		\end{align}
		Combining (\ref{canonicalcurv_identity_a}) and (\ref{canonicalcurv_identity_b}) then gives the result in this case:
		\begin{align*}
			\sum_{i=1}^{4n-1}R(X,Y,e_s,\varphi_i(e_s)) &= [(4n-4)\alpha\beta +4\alpha\beta]\Phi_i(X,Y) = 4n\alpha\beta \Phi_i(X,Y).
		\end{align*}
		\item For the mixed cases, the first paragraph of \cite[Section 2.2]{ADScurv} shows that $R(X,Y)$ is the zero operator when $X\in \mathcal{H},Y\in\mathcal{V}$ or $X\in\mathcal{V}$, $Y\in\mathcal{H}$.
		\item Suppose now that $X,Y\in\mathcal{V}$, and without loss of generality write $X=\xi_p$, $Y=\xi_q$ for $(p,q,r)$ an even permutation of $(1,2,3)$. Letting $e_s\in\mathcal{H}$, it follows from \cite[Eqns.\@ (2.4), (2.5)]{ADScurv} respectively that
		\begin{align*}
			R(\xi_p,\xi_q,\xi_j,\xi_k) &= -4\alpha\beta (\delta_{p,j}\delta_{q,k} -\delta_{p,k}\delta_{q,j}) = -4\alpha\beta (\eta_p\wedge \eta_q)(\xi_j,\xi_k) = 4\alpha\beta \Phi_r(\xi_j,\xi_k) ,    \\
			R(\xi_p,\xi_q,e_s,\varphi_ie_s) &= 2\alpha\beta \Phi_r(e_s, \varphi_ie_s) ,
		\end{align*}
		and combining these gives
		\begin{align*}
			\sum_{s=1}^{4n-1} R(X,Y,e_s,\varphi_i (e_s)) &=\sum_{s=1}^{4n-1} R(\xi_p,\xi_q,e_s,\varphi_i e_s) = 2[4\alpha\beta \Phi_r(\xi_j,\xi_k) ] +2\alpha\beta\sum_{s=4}^{4n-1} \Phi_r(e_s,\varphi_i e_s) \\
			&= 8\alpha\beta \Phi_r(\xi_j,\xi_k) - 2 \alpha\beta(4n-4) \delta_{i,r} = -8\alpha\beta \delta_{i,r}- 8(n-1)\alpha\beta\delta_{i,r}\\
			&= -8n\alpha\beta \delta_{i,r} = 8n\alpha\beta \Phi_i(\xi_p,\xi_q) = 8n\alpha\beta \Phi_i(X,Y),
		\end{align*}
		where we have calculated using an even permutation $(i,j,k)$ of $(1,2,3)$.
	\end{enumerate}
	\end{proof}
We now give formulas for the $\varphi_i$-twisted traces of the torsion terms appearing in the formula (\ref{curvdifference}) relating $R$ and $R^g$:
\begin{proposition} If $(M,g,\xi_i,\eta_i,\varphi_i)$ is a $\tad$ manifold and $e_1,\dots,e_{4n-1}$ an adapted orthonormal (local) frame, then \label{conversions}
	\begin{align*}
		\sum_{s=1}^{4n-1}g(T(X,Y),T(e_s,\varphi_i(e_s))) &= \begin{cases}
			\{-16(n-1)\alpha^2 +8\alpha(\delta-4\alpha)\}\Phi_i(X,Y)  & X,Y\in \mathcal{H}, \\
			0& X\in \mathcal{H},Y\in\mathcal{V},\\
			0 & X\in \mathcal{V},Y\in\mathcal{H},\\
			8(\delta-4\alpha) (2(n+1)\alpha-\delta)\Phi_i(X,Y) & X,Y\in\mathcal{V},
		\end{cases}
	\end{align*}
	and \begin{align*}
		\sum_{s=1}^{4n-1} dT(X,Y,e_s,\varphi_ie_s) &= \begin{cases}
			\{ -16\alpha^2 (2n-1) +8\alpha\beta\} \Phi_i(X,Y) & X,Y\in \mathcal{H}, \\
			0& X\in \mathcal{H},Y\in\mathcal{V},\\
			0 & X\in \mathcal{V},Y\in\mathcal{H},\\
			32(n-1)\alpha(\delta-2\alpha)\Phi_i(X,Y) & X,Y\in\mathcal{V}.
		\end{cases}
	\end{align*}
\end{proposition}
\begin{proof}
	From \cite[Eqns.\@ (1.10), (1.11), (1.12)]{ADScurv}, the canonical torsion and its exterior derivative satisfy 
		\begin{align}
			T(\xi_j,\xi_k)&=2(\delta-4\alpha)\xi_i\label{torsionformulaver}\\
			T(X,Y)&= 2\alpha\sum_{p=1}^3 [\eta_p(Y)\varphi_p(X) -\eta_p(X)\varphi_p(Y)+\Phi_p(X,Y)\xi_p ] - 2(\alpha-\delta)\mathfrak{S}_{i,j,k}\eta_{ij}(X,Y)\xi_k, \label{torsionformula}\\
			dT&= 4\alpha^2\sum_{p=1}^3 \Phi_p^{\mathcal{H}}\wedge \Phi_p^{\mathcal{H}} +8\alpha(\delta-2\alpha)\mathfrak{S}_{i,j,k} \Phi_i^{\mathcal{H}} \wedge \eta_{jk} . \label{dtorsion} 
		\end{align}
	We treat the cases one at a time: 
	\begin{enumerate}[(i)] 
		\item Suppose first that $X,Y\in\mathcal{H}$. If $e_s\in\mathcal{H}$, then substituting (\ref{torsionformulaver}) and (\ref{torsionformula}) and gives
			\begin{align*}
				g(T(X,Y), T(\xi_j,\xi_k))&= g\left( 2\alpha\sum_{p=1}^3 \Phi_p(X,Y)\xi_p,\ 2(\delta-4\alpha)\xi_i \right) = 4\alpha(\delta-4\alpha) \Phi_i(X,Y) ,\\
				g(T(X,Y),T(e_s,\varphi_ie_s)) &= g\left( 2\alpha\sum_{p=1}^3 \Phi_p(X,Y)\xi_p,\  2\alpha\sum_{l=1}^3 \Phi_l(e_s,\varphi_ie_s) \xi_l \right) = -4\alpha^2 \Phi_i(X,Y),
			\end{align*}
		and hence 
		\begin{align*}
			\sum_{s=1}^{4n-1} g(T(X,Y),T(e_s,\varphi_ie_s)) &= (4n-4)[-4\alpha^2\Phi_i(X,Y)] + 2[4\alpha(\delta-4\alpha)\Phi_i(X,Y)] \\
			&=\{-16(n-1)\alpha^2 +8\alpha(\delta-4\alpha)\}\Phi_i(X,Y) .
		\end{align*}
		The formula $$\sum_{s=1}^{4n-1} dT(X,Y,e_s,\varphi_ie_s) = \{ -16\alpha^2 (2n-1) +8\alpha\beta\} \Phi_i(X,Y) $$ appears as \cite[Eqn.\@ (3.4)]{ADScurv} (note that we are working in dimension $4n-1$ rather than $4n+3$). 
		\item Suppose that $X\in\mathcal{H}$, $Y\in\mathcal{V}$ or $X\in\mathcal{V}$, $Y\in\mathcal{H}$. If $e_s\in\mathcal{H}$, then substituting (\ref{torsionformulaver}) and (\ref{torsionformula}) gives
		\begin{align*}
			g(T(X,Y),T(\xi_j,\xi_k)) &= g\left( 2\alpha\sum_{p=1}^3 [\eta_p(Y)\varphi_pX - \eta_p(X)\varphi_pY] ,\ 2(\delta-4\alpha)\xi_i \right) \\
			&= 4\alpha (\delta-4\alpha)\sum_{p=1}^3 [\eta_p(Y)g(\varphi_pX,\xi_i)-\eta_p(X)g(\varphi_pY,\xi_i)] =0
		\end{align*}
		and
		\begin{align*}
			g(T(X,Y),T(e_s,\varphi_ie_s)) &= g\left( 2\alpha\sum_{p=1}^3 [\eta_p(Y)\varphi_pX - \eta_p(X)\varphi_pY] ,\ 2\alpha\sum_{l=1}^3 \Phi_l(e_s,\varphi_ie_s)\xi_l \right) \\
			&= -4\alpha^2 \sum_{p=1}^3 [ \eta_p(Y)g(\varphi_p X,\xi_i) - \eta_p(X)g(\varphi_pY,\xi_i)]= 0 .  
		\end{align*}
		The formula (\ref{dtorsion}) immediately implies that $dT(X,Y,e_s,\varphi_ie_s) = dT(X,Y,\xi_j,\xi_k)=0$. Both of the desired formulas then follow by taking sums.
		\item Suppose that $X,Y\in\mathcal{V}$. If $e_s\in\mathcal{H}$, then substituting (\ref{torsionformulaver}) and (\ref{torsionformula}) gives
		\begin{align*}
			&g(T(X,Y),T(\xi_j,\xi_k)) = 4\alpha(\delta-4\alpha) \sum_{p=1}^3[\eta_p(Y) g(\varphi_pX,\xi_i) - \eta_p(X)g(\varphi_pY,\xi_i)] \\&\qquad \qquad  +4\alpha(\delta-4\alpha) \Phi_i(X,Y) -4(\alpha-\delta)(\delta-4\alpha) \eta_{jk}(X,Y)\\
			=& 4\alpha(\delta-4\alpha)[2\Phi_i(X,Y)] +4\alpha(\delta-4\alpha) \Phi_i(X,Y)  +4(\alpha-\delta)(\delta-4\alpha) \Phi_i(X,Y)\\
			=&-4(\delta-4\alpha)^2 \Phi_i(X,Y)  
		\end{align*}
		and
		\begin{align*}
		&	g(T(X,Y),T(e_s,\varphi_ie_s)) = g(T(X,Y), -2\alpha \xi_i) \\
			=& -4\alpha^2\sum_{p=1}^3 [\eta_p(Y) g(\varphi_pX,\xi_i) - \eta_p(X)g(\varphi_pY,\xi_i)]  -4\alpha^2 \Phi_i(X,Y) +4\alpha(\alpha-\delta) \eta_{jk}(X,Y)\\
			=& -4\alpha^2[2\Phi_i(X,Y)] -4\alpha^2 \Phi_i(X,Y) - 4\alpha(\alpha-\delta)\Phi_i(X,Y)= [-16\alpha^2+4\alpha\delta]\Phi_i(X,Y),
		\end{align*}
		and it follows that
			\begin{align*}
				\sum_{s=1}^{4n-1} g(T(X,Y),T(e_s,\varphi_ie_s)) &= (4n-4)[-16\alpha^2 +4\alpha\delta]\Phi_i(X,Y) + 2[-4(\delta-4\alpha)^2 \Phi_i(X,Y)] \\
				&= 8(\delta-4\alpha) (2(n+1)\alpha-\delta)\Phi_i(X,Y).
			\end{align*}
		On the other hand, if $e_s\in\mathcal{H}$, then (\ref{dtorsion}) gives
		\begin{align*}
			dT(X,Y,e_s,\varphi_ie_s)&= 8\alpha (\delta-2\alpha) \mathfrak{S}_{p,q,r} \Phi_p^{\mathcal{H}}(e_s,\varphi_ie_s)\eta_{qr}(X,Y) = 8\alpha (\delta-2\alpha) \Phi_i(X,Y) ,\\
			dT(X,Y,\xi_j,\xi_k)&= 0,
		\end{align*}
		and hence
			\begin{align*}
				\sum_{s=1}^{4n-1} dT(X,Y,e_s,\varphi_ie_s) &= (4n-4)[8\alpha (\delta-2\alpha)\Phi_i(X,Y)] +2[0]  = 32(n-1)\alpha(\delta-2\alpha)\Phi_i(X,Y)  .
		\end{align*}
	\end{enumerate}
\end{proof}
Having calculated the $\varphi_i$-twisted traces of the terms on the right hand side of (\ref{curvdifference}), we now have a formula for the $\varphi_i$-twisted traces of $R^g$. These will be used in the next subsection to deduce explicit formulas for the projections of the Levi-Civita spinorial curvatures onto the bundles $E_i$, which will be the crucial ingredient in the proof that $\widehat{\nabla}\rvert_E$ is flat.

\subsection{Projection Identities and Flatness of the Modified Connection}
In the present subsection we show, using the curvature identities calculated in the previous subsection, that the restriction of the modified connection $\widehat{\nabla}$ to $E=E_1+E_2+E_3$ (the non-direct sum) is flat. Assuming the manifold is simply-connected we obtain, as a consequence of flatness, a basis of $E$ consisting of globally-defined $\widehat{\nabla}$-parallel spinors, proving Theorem \ref{deformedKillingspinorstheorem}. We begin with the $E_i$-projection formulas for the Levi-Civita curvature: 

\begin{proposition} \label{LCprojections}  Let $(M,g,\xi_i,\eta_i,\varphi_i)$ be a $\tad$ manifold and suppose that $\psi\in\Gamma(E_i)$. For any $X,Y\in TM$, the orthogonal projection onto $E_i$ of the Levi-Civita spinorial curvature $R^g(X,X)\psi$ is given by
		\[
		\text{\emph{pr}}_{E_i} R^g(X,Y)\psi = \begin{cases}
			\{-2(n-1)\alpha(\alpha-\delta)+\frac{1}{2}\delta^2\}\Phi_i(X,Y)\ \xi_i\cdot \psi  &X,Y\in\mathcal{V}, \\
			\{(2n-1)\alpha\delta -(2n-\frac{3}{2})\alpha^2\} \Phi_i(X,Y)\ \xi_i\cdot \psi  &X,Y\in\mathcal{H},\\
			0 &\text{$X\in \mathcal{H},Y\in\mathcal{V}$ \emph{or} $X\in\mathcal{V},Y\in\mathcal{H}$.}
		\end{cases} 
		\]
\end{proposition}
\begin{proof}
	Letting $e_1,\dots, e_{4n-1}$ be an adapted orthonormal frame, we recall from the proof of \cite[Thm.\@ 1]{Fried90} that if $\psi\in\Gamma(E_i)$ then $e_p\cdot e_q\cdot \psi$ is orthogonal to $E_i$ unless $e_q = e_p$ or $\pm \varphi_ie_p$. We also recall that the defining condition for $E_i$ implies that if $\psi\in\Gamma(E_i)$ then $e_s\cdot \varphi_i(e_s)\cdot \psi = \xi_i \cdot \psi$ (for $e_s\neq \xi_i$). Using (\ref{curvdifference}), we have:
		\begin{align*}
			&\text{pr}_{E_i}R^g(X,Y)\psi = \frac{1}{4}\sum_{s=1}^{4n-1} R^g(X,Y,e_s,\varphi_i(e_s)) \ e_s\cdot \varphi_i(e_s)\cdot \psi= \frac{1}{4}\sum_{s=1}^{4n-1} R^g(X,Y,e_s,\varphi_i(e_s)) \ \xi_i\cdot \psi \\
			&\qquad = \frac{1}{4}\sum_{s=1}^{4n-1}\left( R(X,Y,e_s,\varphi_i(e_s)) -\frac{1}{4} g(T(X,Y),T(e_s,\varphi_i(e_s))) -\frac{1}{8}dT(X,Y,e_s,\varphi_i(e_s))\right)  \xi_i\cdot \psi .
		\end{align*}
	The result then follows by substituting the expressions from Propositions \ref{curvprop} and \ref{conversions}.
\end{proof}

We also need to compute various identities for the $E_i$-projections of terms involving the fundamental $2$-forms $\Phi_i$, $i=1,2,3$, as such terms will arise in the spinorial curvatures of $\widehat{\nabla}$ due to the final term on the right hand side of (\ref{modified_connection_formula}). To that end, the following lemma collects the required projection formulas:
\begin{lemma} Let $(M,g,\xi_i,\eta_i,\varphi_i)$ be a $\tad$ manifold. For any $\psi\in\Gamma(E_i)$ and any even permutation $(p,q,r)$ of $(1,2,3)$, we have \label{projectionlemma}
	\begin{enumerate}[(i)]
		\item $\text{\emph{pr}}_{E_i}(\Phi_p\cdot \psi) = -\delta_{i,p}(2n-1)\xi_i\cdot \psi$,
		\item $\text{\emph{pr}}_{E_i}(\Phi_p\cdot\Phi_q\cdot \psi -\Phi_q\cdot \Phi_p\cdot \psi)= 2\delta_{i,r} (4n-3)\xi_i\cdot \psi $,
		\item $\text{\emph{pr}}_{E_i}((\nabla^g_{\xi_p}\Phi_q)\cdot \psi - (\nabla_{\xi_q}^g\Phi_p)\cdot \psi  ) = \delta_{i,r}[-2\delta +8(n-1)(\alpha-\delta)] \xi_i\cdot \psi $.
	\end{enumerate}
\end{lemma}
\begin{proof}
	Letting $\psi\in \Gamma(E_i)$, we prove the three identities one at a time.
	\begin{enumerate}[(i)]
		\item This follows by writing $\Phi_p=-\frac{1}{2} \sum_{s=1}^{4n-1} e_s\wedge \varphi_p(e_s)$ in an adapted frame and using 
		\[
		\text{proj}_{E_i^-} (e_s\cdot \varphi_p(e_s)\cdot \psi) = \delta_{i,p} \  \xi_i\cdot \psi 
		\]
		(see the proof of \cite[Thm.\@ 1]{Fried90}).
		\item We use the relation $V\lrcorner \Phi_q =-\frac{1}{2}(V\cdot \Phi_q - \Phi_q\cdot V)$, which may be deduced by subtracting Equations (1.4) in Chapter 1.2 of \cite{BFGK}. Considering first the horizontal part of $\Phi_p$, we calculate
		\begin{align*}
			\Phi_p^{\mathcal{H}}\cdot \Phi_q\cdot \psi &= -\frac{1}{2} \sum_{s=4}^{4n-1} e_s\cdot \varphi_p(e_s)\cdot \Phi_q\cdot \psi =-\frac{1}{2}\sum_{s=4}^{4n-1} e_s\cdot[-2(\varphi_p(e_s)\lrcorner \Phi_q) + \Phi_q \cdot \varphi_p(e_s) ]\cdot \psi \\
			&= -\frac{1}{2} \sum_{s=4}^{4n-1} [-2e_s\cdot \varphi_r(e_s) +e_s\cdot \Phi_q\cdot \varphi_p(e_s)]\cdot \psi \\
			&= -\frac{1}{2} \sum_{s=4}^{4n-1} [-2e_s\cdot \varphi_r(e_s) +(-2e_s\lrcorner \Phi_q + \Phi_q\cdot e_s   )   \cdot \varphi_p(e_s)]\cdot \psi \\
			&= -2\Phi_r^{\mathcal{H}}\cdot \psi  +\Phi_q\cdot \Phi_p^{\mathcal{H}} \cdot \psi +\sum_{s=4}^{4n-1} \varphi_p(e_s)\cdot \varphi_q(e_s) \cdot \psi ,
		\end{align*}
		and similarly for the vertical part,
		\begin{align*}
			\Phi_p^{\mathcal{V}} \cdot \Phi_q\cdot \psi &= -\xi_q\cdot \xi_r\cdot \Phi_q\cdot \psi = -\xi_q \cdot( -2\xi_r\lrcorner \Phi_q + \Phi_q\cdot \xi_r   )\cdot \psi =-\xi_q \cdot (2\xi_p +\Phi_q\cdot \xi_r)\cdot \psi\\
			&= -2\xi_q\cdot \xi_p \cdot \psi + (2\xi_q\lrcorner\Phi_q -\Phi_q\cdot \xi_q\cdot \xi_r )\cdot \psi \\
			&= -2\Phi_r^{\mathcal{V}}\cdot \psi +\Phi_q\cdot \Phi_p^{\mathcal{V}} \cdot \psi ,   
		\end{align*}
		Adding the above two equations, we deduce
		\begin{align*}
		(\Phi_p\cdot \Phi_q - \Phi_q\cdot \Phi_p) \cdot \psi = -2\Phi_r\cdot \psi +\sum_{s=4}^{4n-1} \varphi_p(e_s)\cdot \varphi_q(e_s)\cdot \psi  ,
		\end{align*}
		and projecting onto $E_i$ using part (i) of this lemma gives the result. 
		\item From Proposition \ref{LCder} we have
		\begin{align*}
			(\nabla^g_{\xi_p}\varphi_q - \nabla^g_{\xi_q}\varphi_p ) &= 2(2\alpha-\delta)\eta_p\otimes \xi_q - 2(2\alpha-\delta) \eta_q\otimes \xi_p -4(\alpha-\delta)\varphi_r \\
			&= 2(2\alpha-\delta) \varphi_r\rvert_{\mathcal{V}} -4(\alpha -\delta) \varphi_r
		\end{align*}
		The result then follows by lowering indices, taking the Clifford product with $\psi$, and projecting onto $E_i$ using part (i).
	\end{enumerate}
\end{proof}

The final step in the proof of Theorem \ref{deformedKillingspinorstheorem} is the following proposition, which asserts that the restriction of $\widehat{\nabla}$ to the bundle $E$ is flat. The theorem will then follow by the holonomy principle (see e.g.\@ \cite[p.\@ 282]{Besse_Einstein_manifolds}).
\begin{proposition}
	The restriction $\widehat{\nabla}\rvert_{E}$ is flat, i.e. $\widehat{R}(\cdot, \cdot)\psi \equiv 0 $ for all $\psi\in \Gamma(E)$.
\end{proposition}
\begin{proof}
	Suppose that $\psi\in\Gamma(E_i)$. Using the definition of $\widehat{\nabla}$ we calculate, for any vector fields $X,Y$,
		\begin{align*}
		&	\widehat{\nabla}_X\widehat{\nabla}_Y \psi = \widehat{\nabla}_X[\nabla_Y^g\psi - \frac{\alpha}{2} Y\cdot \psi - \frac{\alpha-\delta}{2} \sum_{p=1}^3\eta_p(Y)\Phi_p\cdot \psi     ]\\
			&= \nabla_X^g\nabla_Y^g\psi -\frac{\alpha}{2} X\cdot \nabla_Y^g\psi -\frac{\alpha-\delta}{2}\sum_{p=1}^3[  \eta_p(X) \Phi_p\cdot \nabla_Y^g\psi ] - \frac{\alpha}{2}\nabla^g_X(Y\cdot \psi) + \frac{\alpha^2}{4} X\cdot Y\cdot \psi\\ 
			&\qquad \qquad  + \frac{\alpha(\alpha-\delta)}{4} \sum_{p=1}^3 [\eta_p(X) \Phi_p\cdot Y\cdot \psi ] -\frac{\alpha-\delta}{2} \sum_{p=1}^3\nabla^g_X[\eta_p(Y)\cdot\Phi_p \cdot \psi]\\
			&\qquad \qquad +\frac{\alpha(\alpha-\delta)}{4} \sum_{p=1}^3 [\eta_p(Y)X\cdot \Phi_p\cdot \psi   ] + \frac{(\alpha-\delta)^2}{4} \sum_{p,q=1}^3[ \eta_p(Y)\eta_q(X) \Phi_q\cdot \Phi_p\cdot \psi ] .\\
		\end{align*}
		Expanding the derivatives and using the formula (\ref{LC_derivative_xii_formula}) for the Levi-Civita derivatives of the Reeb vector fields (or equivalently, the dual $1$-forms) shows that the above expression is equal to 
		\begin{align*}
			& \nabla_X^g\nabla_Y^g\psi -\frac{\alpha}{2} X\cdot \nabla_Y^g\psi -\frac{\alpha-\delta}{2}\sum_{p=1}^3[  \eta_p(X) \Phi_p\cdot \nabla_Y^g\psi ] - \frac{\alpha}{2}[(\nabla^g_XY)\cdot \psi +  Y\cdot\nabla_X^g\psi] \\
			& + \frac{\alpha^2}{4} X\cdot Y\cdot \psi+ \frac{\alpha(\alpha-\delta)}{4} \sum_{p=1}^3 [\eta_p(X) \Phi_p\cdot Y\cdot \psi ]\\ 
			&   -\frac{\alpha-\delta}{2} \sum_{p=1}^3\bigg[   \big(\alpha\Phi_p(X,Y)+(\alpha-\delta) \eta_{p+1,p+2}(X,Y) +\eta_p(\nabla_X^gY) \big)  \Phi_p\cdot \psi  + \eta_p(Y)(\nabla_X^g\Phi_p)\cdot \psi \\ & +\eta_p(Y) \Phi_p\cdot \nabla_X^g\psi       \bigg] 
			+\frac{\alpha(\alpha-\delta)}{4} \sum_{p=1}^3 [\eta_p(Y)X\cdot \Phi_p\cdot \psi   ] + \frac{(\alpha-\delta)^2}{4} \sum_{p,q=1}^3[ \eta_p(Y)\eta_q(X) \Phi_q\cdot \Phi_p\cdot \psi ].
		\end{align*}
	The spinorial curvature is then given by
		\begin{align*}
			&\widehat{R}(X,Y)\psi = \widehat{\nabla}_X\widehat{\nabla}_Y \psi - \widehat{\nabla}_Y\widehat{\nabla}_X\psi - \widehat{\nabla}_{[X,Y]}\psi \nonumber \\
			&= R^g(X,Y)\psi +\frac{\alpha^2}{4} (X\cdot Y-Y\cdot X)\cdot \psi +\frac{\alpha(\alpha-\delta)}{2} \sum_{p=1}^3 [\eta_p(X)( Y\lrcorner \Phi_p)    -  \eta_p(Y)(X\lrcorner \Phi_p)   ]\cdot \psi \nonumber \\
			&- (\alpha-\delta)\sum_{p=1}^3[\alpha\Phi_p(X,Y) +(\alpha-\delta) \eta_{p+1,p+2}(X,Y)]\Phi_p\cdot \psi  -\frac{\alpha-\delta}{2}\sum_{p=1}^3 [\eta_p(Y)(\nabla^g_X\Phi_p)  -\eta_p(X)(\nabla^g_Y\Phi_p)]\cdot \psi \nonumber\\
			&  +\frac{(\alpha-\delta)^2}{4} \sum_{p,q=1}^3 [\eta_p(Y)\eta_q(X)-\eta_p(X)\eta_q(Y)]\cdot\Phi_q\cdot \Phi_p \cdot \psi ,
		\end{align*}
	where the indices $p,p+1,p+2$ are taken modulo 3. The result now follows by considering the various cases of $X,Y$ being in $\mathcal{H},\mathcal{V}$ and projecting onto $E_i$, using the formulas from Proposition \ref{LCprojections} and Lemma \ref{projectionlemma}.
\end{proof}

\begin{remark}\label{Dirac_remark}
	Finally, we calculate the action of the (Riemannian) Dirac operator on $\mathcal{H}$-Killing spinors. If $\psi\in\Gamma(E)$ is an $\mathcal{H}$-Killing spinor, then the Riemannian Dirac operator acts on it by
	\begin{align*}
		D\psi &= \sum_{s=1}^{4n-1} e_s\cdot \nabla^g_{e_s}\psi = \sum_{s=1}^{4n-1} e_s\cdot \left( \frac{\alpha}{2} e_s\cdot \psi + \frac{\alpha-\delta}{2}\sum_{p=1}^{3} \eta_p(e_s) \Phi_p\cdot \psi \right) \\
		&= -\frac{(4n-1)\alpha}{2} \psi + \frac{\alpha-\delta}{2} \sum_{p=1}^3 \xi_p\cdot \Phi_p \cdot \psi .
	\end{align*}
This shows $\mathcal{H}$-Killing spinors satisfy an eigenvalue equation up to a correction term, which vanishes in the $3$-$\alpha$-Sasakian case (consistent with the fact that these are classical Killing spinors when $\alpha=\delta$).
\end{remark}
%
%
%
%
%

\section{Compact/Noncompact Dual Pairs of Homogeneous 3-$(\alpha,\delta)$-Sasaki Spaces}\label{section:duality}

Let us begin by briefly recalling the notion of 3-Sasakian data from \cite{homdata}, which the authors of that paper use to characterize homogeneous 3-Sasakian manifolds:
\begin{definition}\label{Chap_dual:3sasakian_data_definition}
	(\cite[Def.\@ 4.1]{homdata}). A 3-Sasakian data consists of a pair of Lie algebras $(\mathfrak{g},\mathfrak{h})$ such that,
	\begin{enumerate}[(i)]
		\item The Lie algebra $\mathfrak{g}$ is compact and simple;
		\item There is a $\Z_2$-graded decomposition $\mathfrak{g}=\mathfrak{g}_0\oplus \mathfrak{g}_1$ with $\mathfrak{g}_0=\mathfrak{h}\oplus \mathfrak{sp}(1)$ a sum of commuting subalgebras; 
		\item There exists an $\mathfrak{h}^{\C}$-module $U$ of complex dimension $2(n-1)$ such that $\mathfrak{g}_1^{\C}\cong \C^2 \otimes U$ as a module for $\mathfrak{g}_0^{\C}\cong \mathfrak{sp}(1)^{\C}\oplus \mathfrak{h}^{\C}$, where $\C^2$ is the standard representation of $\mathfrak{sp}(1)^{\C} \cong \mathfrak{sl}(2,\C)$. 
	\end{enumerate} 
\end{definition} 
In \cite{hom3alphadelta} this same data without the compactness assumption is called \emph{generalized 3-Sasakian data}, and the authors note that it specifies families of $\tad$ homogeneous spaces. 

\subsection{Duality of Extended Symmetric Data}

What follows is a dualization procedure generalizing the homogeneous $\tad$ duality described in \cite[Remark 3.1.1(c)]{hom3alphadelta}. Generalizing Definition \ref{Chap_dual:3sasakian_data_definition}, we define:
\begin{definition}\label{Chap_dual:extended_symmetric_data_definition}
	\emph{Extended symmetric data} $(\mathfrak{g},\mathfrak{h},\mathfrak{k},g)$ consists of a triple of real Lie algebras with $\mathfrak{h},\mathfrak{k}\subset \mathfrak{g}$, together with an inner product $g$ on $\mathfrak{g}/\mathfrak{h}$, such that the following properties hold: 
	\begin{enumerate}[(i)]
		\item The Lie algebra $\mathfrak{g}$ is semi-simple;
		\item There is a $\Z_2$-grading $\mathfrak{g}=\mathfrak{g}_0\oplus \mathfrak{g}_1$ such that $\mathfrak{g}_0=\mathfrak{h}\oplus \mathfrak{k}$;
		\item  The inner product $g$ is positive definite and takes the form
		\begin{align} g = \lambda_0\kappa_{\mathfrak{g}}\rvert_{\mathfrak{k}\times\mathfrak{k}} +\lambda_1\kappa_{\mathfrak{g}}\rvert_{\mathfrak{g}_1\times\mathfrak{g}_1}, \qquad \lambda_0,\lambda_1\in\R\setminus \{0\}, \label{Chap_dual:metricform}
		\end{align}
		under the natural identification $\mathfrak{g}/\mathfrak{h}\cong \mathfrak{k}\oplus \mathfrak{g}_1$, where $\kappa_{\mathfrak{g}}$ denotes the Killing form of $\mathfrak{g}$.
	\end{enumerate}
\end{definition}
The idea behind the preceding definition is that the Lie algebras $(\mathfrak{g},\mathfrak{g}_0)$ constitute a Riemannian symmetric pair, thus the pair $(\mathfrak{g},\mathfrak{h})$ can be viewed as the Lie algebraic data of a homogeneous space fibering over a symmetric base. Indeed, the $\Z_2$-grading $\mathfrak{g}=\mathfrak{g}_0\oplus \mathfrak{g}_1$, together with the fact that $\mathfrak{k}$ and $\mathfrak{h}$ commute, gives the following commutator relations:
\begin{align} [\mathfrak{h},\mathfrak{h}]\subseteq \mathfrak{h}, \quad [\mathfrak{k},\mathfrak{k}]\subseteq \mathfrak{k}, \quad[\mathfrak{h},\mathfrak{k}]=0, \quad  [\mathfrak{g}_0,\mathfrak{g}_1]\subseteq \mathfrak{g}_1,  \quad  [\mathfrak{g}_1,\mathfrak{g}_1] \subseteq  \mathfrak{h}\oplus \mathfrak{k} . \label{Chap_dual:commutatorrels} 
\end{align}

The case of most interest for us is that of $\ts$ data as in Definition \ref{Chap_dual:extended_symmetric_data_definition}, though we give a construction which is valid in the slightly more general setting of extended symmetric data.
\begin{remark}
	Associated to extended symmetric data $(\mathfrak{g},\mathfrak{h},\mathfrak{k},g)$ is the vector space 
	\[
	\mathfrak{m}:= \mathfrak{k}\oplus \mathfrak{g}_1
	\]
	serving as a reductive complement to $\mathfrak{h} \subseteq \mathfrak{g}$. Conversely, the Lie algebra $\mathfrak{k}$ may be recovered from $\mathfrak{m}$ via $\mathfrak{k} = \mathfrak{m}\cap \mathfrak{g}_0$. For this reason we shall also refer to $(\mathfrak{g},\mathfrak{h},\mathfrak{m},g)$ as extended symmetric data. We shall also denote by $\mathfrak{m}_i:= \mathfrak{m}\cap \mathfrak{g}_i$, $i=0,1$ the components of $\mathfrak{m}$ with respect to the $\Z_2$-grading on $\mathfrak{g}$.
\end{remark}
With the preceding remark in mind, we are ready to define our notion of duality at the Lie algebra level:
\begin{definition} 
	\label{Chap_dual:Liedualdef}Given extended symmetric data $(\mathfrak{g},\mathfrak{h},\mathfrak{m},g)$, we define (for the same $\mathfrak{k}$) the dual extended symmetric data $(\mathfrak{g}', \mathfrak{h}, \mathfrak{m}', g')$ by setting
	\begin{align*} \mathfrak{g}'&:= \mathfrak{g}_0 \oplus i\mathfrak{g}_1 \subseteq \mathfrak{g}^{\C} , \\ \mathfrak{m}'&:= \mathfrak{k} \oplus i\mathfrak{g}_1\subseteq \mathfrak{g}' ,  \end{align*}
	and taking $g'$ to be the real inner product induced on $\mathfrak{m}'$ by extending $g$ sesquilinearly to $\mathfrak{g}^{\C}$ and restricting to the real form $\mathfrak{g}'\subset \mathfrak{g}^{\C}$. 
\end{definition}
The extension of $g$ to $\mathfrak{g}^{\C}$ is done sesquilinearly in the preceding definition in order to ensure that $g'$ is positive definite--exactly as in the usual duality of symmetric spaces. It is furthermore clear that this dualization procedure is involutive, so it makes sense to consider pairs of dual extended symmetric spaces without specifying in which direction the duality goes. 
\begin{remark}
Let us briefly compare Definition \ref{Chap_dual:Liedualdef} with Kath's notion of duality in \cite{kath_Tduals}. Apart from the fact that Kath's duality is between Riemannian and pseudo-Riemannian spaces, we note that it also depends on a Lie algebra involution $T$, whose $1$-eigenspace (resp. $(-1)$-eigenspace) indicates which tangent directions on the compact side should correspond to tangent directions on the non-compact side with positive norm squared (resp. negative norm squared). Such an involution is not necessary for our purposes, as the decomposition $\mathfrak{m}=\mathfrak{k}\oplus \mathfrak{g}_1$ automatically keeps track of which directions are to be modified. Rather, our duality construction is obtained from the duality between compact and non-compact symmetric spaces by dualizing the symmetric pair $(\mathfrak{g},\mathfrak{g}_0)$.
\end{remark}
\begin{definition}
	Let $K$ be a connected Lie group with Lie algebra $\mathfrak{k}$. An \emph{extended symmetric space} (relative to $K$) is a connected Riemannian homogeneous space $(M:=G/H,g)$ with Lie algebra decomposition as in Definition \ref{Chap_dual:extended_symmetric_data_definition}. Letting $G'$ be the connected subgroup of $G^{\C}$ corresponding to the real Lie subalgebra $\mathfrak{g}' \subseteq \mathfrak{g}^{\C}$, we define the dual of $(M,g)$ to be $(M':=G'/H, g')$. 
\end{definition}
	In order to investigate the spinorial properties of the dual we first describe the special orthogonal group $\SO(\mathfrak{m}',g')$ of the Riemannian metric $g'$. We define 
\begin{align*} 
	\mathfrak{so}(\mathfrak{m})_0 &:= \{A\in \mathfrak{so}(\mathfrak{m}): A(\mathfrak{k})\subseteq \mathfrak{k}\text{ and } A(\mathfrak{m}_1)\subseteq \mathfrak{m}_1    \} , \\
	\mathfrak{so}(\mathfrak{m})_1 &:= \{B\in \mathfrak{so}(\mathfrak{m}): B(\mathfrak{k})\subseteq \mathfrak{m}_1\text{ and } B(\mathfrak{m}_1)\subseteq \mathfrak{k}    \} , 
\end{align*} 
and the non-standard Lie bracket $[[\cdot,\cdot]]$ on $ \mathfrak{so}(\mathfrak{m})_0\oplus i\mathfrak{so}(\mathfrak{m})_1$ given by
\[
 [[A_1,A_2]]:= [A_1,A_2]_{\mathfrak{so}(\mathfrak{m})^{\C}}, \quad [[A,iB]]:=i[A,B]_{\mathfrak{so}(\mathfrak{m})^{\C}}, \quad [[iB_1,iB_2]]:=[B_1,B_2]_{\mathfrak{so}(\mathfrak{m})^{\C}}  ,
\]
where $[ \cdot , \cdot ]_{\mathfrak{so}(\mathfrak{m})^{\C}}$ denotes the usual commutator in $\mathfrak{so}(\mathfrak{m})^{\C}$. It is clear that the bracket $[[\cdot,\cdot]]$ is constructed so that $\mathfrak{so}(\mathfrak{m})_0\oplus i\mathfrak{so}(\mathfrak{m})_1$ has the same commutators as $\mathfrak{so}(\mathfrak{m})= \mathfrak{so}(\mathfrak{m})_0\oplus \mathfrak{so}(\mathfrak{m})_1$. The following two lemmas and subsequent proposition are analogous to \cite[Props.\@ 6.1, 3.1, 4.1]{kath_Tduals}; we identify $\mathfrak{so}(\mathfrak{m}',g')$ with $\mathfrak{so}(\mathfrak{m})_0\oplus i \mathfrak{so}(\mathfrak{m})_1$, and explain how to use this identification to describe the relationship between the Levi-Civita connections on the dual pair of spaces. 
\begin{lemma}
	Let $(\mathfrak{g},\mathfrak{h},\mathfrak{m},g)$ be extended symmetric data, and $(\mathfrak{g}',\mathfrak{h},\mathfrak{m}',g')$ the dual data. The map $\tau\: \mathfrak{so}(\mathfrak{m})_0\oplus i\mathfrak{so}(\mathfrak{m})_1  \to \mathfrak{so}(\mathfrak{m}',g')$ given by \begin{align*}
		\tau(A)(x) = A(x),\quad \tau(A)(iy)= iA(y),\quad \tau(iB)(x)=iB(x),\quad \tau(iB)(iy)= B(y),
	\end{align*}
	for all $A\in\mathfrak{so}(\mathfrak{m})_0$, $B\in\mathfrak{so}(\mathfrak{m})_1$, $x\in \mathfrak{k}$, $y\in\mathfrak{m}_1$ is an isomorphism of Lie algebras.
\end{lemma}
\begin{proof}
	Let $x_1,x_2\in\mathfrak{k}$, $y_1,y_2\in\mathfrak{m}_1$, $A\in \mathfrak{so}(\mathfrak{m})_0$, and $B\in \mathfrak{so}(\mathfrak{m})_1$. Using the definitions of $\mathfrak{so}(\mathfrak{m})_0$, $\mathfrak{so}(\mathfrak{m})_1$, and $g'$, we calculate
	\begin{align*}
		g'(\tau(A+iB)x_1,x_2) + g'(x_1,\tau(A+iB)x_2) &= g'(Ax_1+iBx_1,x_2) + g'(x_1,Ax_2+iBx_2)\\
		&= g(Ax_1,x_2) + g(x_1,Ax_2) =0 , \\
		g'(\tau(A+iB)x_1,iy_1) + g'(x_1,\tau(A+iB)(iy_1))  &= g'( Ax_1 +iBx_1, iy_1 ) + g'( x_1, iAy_1 +By_1)\\
		&= g(Bx_1,y_1) + g(x_1,By_1) =0 , \\
		g'(\tau(A+iB)(iy_1),iy_2) + g'(iy_1,\tau(A+iB)(iy_2))&= g'(iAy_1+By_1,iy_2) + g'(iy_1, iAy_2 +By_2) \\
		&= g(Ay_1,y_2) +g(y_1,Ay_2) =0,
	\end{align*}
	hence $\tau(A+iB)\in\mathfrak{so}(\mathfrak{m}',g')$. The map $\tau$ is a linear isomorphism, so it remains only to check that it is a Lie algebra homomorphism. One the one hand, we calculate
	\begin{align*}
		\tau [[A+iB,C+iD]] &= \tau( [A,C]_{\mathfrak{so}(\mathfrak{m})^{\C}}+[B,D]_{\mathfrak{so}(\mathfrak{m})^{\C}} + i[B,C]_{\mathfrak{so}(\mathfrak{m})^{\C}}+i[A,D]_{\mathfrak{so}(\mathfrak{m})^{\C}}),
	\end{align*}
	and on the other hand,
	\begin{align*}
		[\tau(A+iB),\tau(C+iD)]_{\mathfrak{so}(\mathfrak{m}',g')}(x)&=(\tau(A+iB)\circ\tau(C+iD))(x) - (\tau(C+iD)\circ\tau(A+iB))(x)  \\
		&= \tau(A+iB)(Cx+iDx) - \tau(C+iD)(Ax+iBx)       \\
		&= ACx +iADx +iBCx +BDx - (CAx +iCBx+iDAx+DBx   )     		\\
		&= \tau([A,C]_{\mathfrak{so}(\mathfrak{m})^{\C}} +i[A,D]_{\mathfrak{so}(\mathfrak{m})^{\C}} +i[B,C]_{\mathfrak{so}(\mathfrak{m})^{\C}} +[B,D]_{\mathfrak{so}(\mathfrak{m})^{\C}})(x) 
	\end{align*}
	in the $\mathfrak{k}$ directions, and
	\begin{align*}
		[\tau(A+iB),\tau(C+iD)]_{\mathfrak{so}(\mathfrak{m}',g')}(iy)&= (\tau(A+iB)\circ\tau(C+iD))(iy) - (\tau(C+iD)\circ\tau(A+iB))(iy )\\
		&=  \tau(A+iB)(iCy+Dy) - \tau(C+iD)(iAy+By)\\
		&=iACy+ADy+BCy+iBDy -(iCAy +CBy+DAy+iDBy)\\
		&=\tau([A,C]_{\mathfrak{so}(\mathfrak{m})^{\C}} +i[A,D]_{\mathfrak{so}(\mathfrak{m})^{\C}} +i[B,C]_{\mathfrak{so}(\mathfrak{m})^{\C}} +[B,D]_{\mathfrak{so}(\mathfrak{m})^{\C}})(iy) 
	\end{align*}
	in the $\mathfrak{m}_1$ directions, completing the proof.
\end{proof}
\begin{lemma}
	\label{Chap_dual:endomorphismparts}Let $(\mathfrak{g},\mathfrak{h},\mathfrak{m},g)$ be extended symmetric data. If $x\in\mathfrak{k}$ and $y\in\mathfrak{m}_1$ then
	\[
	\ad(x)\in\mathfrak{so}(\mathfrak{m})_0,
	\quad   \text{\emph{proj}}_{\mathfrak{m}}\circ\ad(y)\in\mathfrak{so}(\mathfrak{m})_1,
	\quad  \Uplambda^g(x)\in\mathfrak{so}(\mathfrak{m})_0,\quad  \Uplambda^g(y)\in\mathfrak{so}(\mathfrak{m})_1 . 
	\]
\end{lemma}
\begin{proof}
	The first two assertions follow from the commutator relations (\ref{Chap_dual:commutatorrels}).
	For the second two assertions, we use the first two together with standard implicit formula for the Nomizu map from Proposition \ref{prelims:LC_Nomizu_general_formula}. Letting $x,x_1,x_2,x_3\in\mathfrak{k}$ and $y,y_1,y_2,y_3\in\mathfrak{m}_1$ and using (\ref{Utensor}), (\ref{Chap_dual:metricform}), and (\ref{Chap_dual:commutatorrels}), we calculate
	\begin{align*}
		&2g(U(x_1,x_2),y) = g([y,x_1],x_2) + g(x_1,[y,x_2]) = 0+0=0,\\
		&2g(U(x_1,y),x_2) = g([x_2,x_1],y) + g(x_1,[x_2,y]) = 0+0=0,\\
		&2g(U(y,x_1),x_2) = g([x_2,y],x_1) + g(y, [x_2,x_1]) = 0+0=0,\\
		&2g(U(y_1,y_2),y_3) = g(\text{proj}_{\mathfrak{m}} [y_3,y_1],y_2) + g( y_1,\text{proj}_{\mathfrak{m}}[y_3,y_2]) = 0+0=0  .
	\end{align*}
	Thus we have proved $(x\lrcorner U) \in \mathfrak{so}(\mathfrak{m})_0$ and $(y\lrcorner U) \in \mathfrak{so}(\mathfrak{m})_1$, and the result follows.
\end{proof}
\begin{proposition}
	\label{Chap_dual:LCNomizu}Let $(M,g)$ be an extended symmetric space, and $(M',g')$ its dual. In terms of the identification $\tau$, the Levi-Civita connection of $(M',g')$ has Nomizu map given by
	\begin{align*} \Uplambda^{g'}(iy) = -\tau(i\Uplambda^g(y))   , \quad \Uplambda^{g'} (x_1)x_2 = \tau(\Uplambda^g(x_1))x_2, \quad \Uplambda^{g'}(x)(iy) = -\tau(\Uplambda^g(x))(iy)  +2i[x,y] . \end{align*}
	for $x\in \mathfrak{k}$, $y\in\mathfrak{m}_1$. 
\end{proposition}
\begin{proof}
	In order to show that the above expression for $\Uplambda^{g'}$ induces a metric connection, it suffices to check that $\Uplambda^{g'}(x)$ is skew-symmetric with respect to $g'$. Using (\ref{Chap_dual:metricform}), (\ref{Chap_dual:commutatorrels}), Lemma \ref{Chap_dual:endomorphismparts}, and the fact that the image of $\tau$ lies in $\mathfrak{so}(\mathfrak{m}',g')$, we calculate:{\small
		\begin{align*}
			&g'(\Uplambda^{g'}(x_1)x_2,x_3) + g'(x_2,\Uplambda^{g'}(x_1)x_3) = g'(\tau(\Uplambda^g(x_1))x_2,x_3) + g'( x_2, \tau(\Uplambda^g(x_1))x_3)=0,\\
			&g'(\Uplambda^{g'}(x_1)x_2,iy) + g'(x_2,\Uplambda^{g'}(x_1)(iy)) =g'( \tau(\Uplambda^g(x_1))x_2, iy) + g'(x_2, -\tau(\Uplambda^g(x_1))(iy)+2i[x_1,y])=0,\\
			&g'(\Uplambda^{g'}(x)(iy_1),iy_2) + g'(iy_1,\Uplambda^{g'}(x)(iy_2)) = g'(-\tau(\Uplambda^g(x))(iy_1) + 2i[x,y_1],iy_2) \\
			&\qquad \qquad \qquad +g'(iy_1,-\tau(\Uplambda^g(x))(iy_2)  + 2i[x,y_2]) = 2\lambda_1 \kappa_{\mathfrak{g}}([x,y_1],y_2) +2\lambda_1 \kappa_{\mathfrak{g}}(y_1,[x,y_2])= 0.
		\end{align*}
	}To see that the given expression for $\Uplambda^{g'}$ is torsion-free, we use the fact that $\Uplambda^g$ is torsion-free to calculate:
	\begin{align*}
		&\Uplambda^{g'}(x_1)x_2 - \Uplambda^{g'}(x_2)x_1 - \text{proj}_{\mathfrak{m}'}[x_1,x_2] =\tau(\Uplambda^g(x_1))x_2 - \tau(\Uplambda^g(x_2))x_1 - [x_1,x_2]\\
		&\qquad \qquad  = \Uplambda^g(x_1)x_2 - \Uplambda^g(x_2)x_1 - [x_1,x_2]=0, \\
		&\Uplambda^{g'}(x)(iy) - \Uplambda^{g'}(iy)(x) - \text{proj}_{\mathfrak{m}'}[x,iy]= -\tau(\Uplambda^g(x))(iy) + 2i[x,y] + \tau(i\Uplambda^g(y))(x) - [x,iy] \\
		&\qquad \qquad = -i\Uplambda^g(x)y +i[x,y] +i\Uplambda^g(y)x = 0,\\
		&\Uplambda^{g'}(iy_1)(iy_2) - \Uplambda^{g'}(iy_2)(iy_1) - \text{proj}_{\mathfrak{m}'}[iy_1,iy_2] = -\tau(i\Uplambda^g(y_1))(iy_2) + \tau(i\Uplambda^g(y_2))(iy_1) +\text{proj}_{\mathfrak{m}}[y_1,y_2]\\
		&\qquad \qquad =-\Uplambda^g(y_1)y_2 + \Uplambda^g(y_2)y_1 + \text{proj}_{\mathfrak{m}}[y_1,y_2]=0 .
	\end{align*}
	The result then follows from the fact that the Levi-Civita connection is the unique torsion-free metric connection.
\end{proof}
Finally we turn our attention to the spinorial properties of the dual pairs. Inspired by \cite[Prop.\@ 7.2]{kath_Tduals}, we show that the question of existence of a homogeneous spin structure is equivalent for dual pairs of extended symmetric spaces:
\begin{proposition}
	If $(M=G/H,g)$ and $(M'=G'/H,g')$ are a dual pair of extended symmetric spaces with connected isotropy group $H$, then $M$ admits a homogeneous spin structure if and only if $M'$ admits one. \label{Chap_dual:dualspinlift}
\end{proposition}
\begin{proof}
	Both directions of the `if and only if' statement are identical, so we prove only the forward direction here. Supposing that $(M=G/H,g)$ admits a homogeneous spin structure, it follows from \cite[Prop.\@ 1.3]{invariantspinstructures} that there is a lift $\widetilde{\Ad}\:H\to\Spin(\mathfrak{m})$ of the isotropy representation $\Ad\:H\to\SO(\mathfrak{m})$, and it suffices to show that $\Ad'\:H\to \SO(\mathfrak{m}')$ lifts to a map $\widetilde{\Ad'}\: H \to\Spin(\mathfrak{m}')$. If $\{v_1,\dots,v_k,w_1,\dots, w_l\}$ is a $g$-orthonormal basis for $\mathfrak{m}$ such that $\{v_t\}_{t=1}^k$ is a basis for $\mathfrak{k}$ and $\{w_t\}_{t=1}^l$ is a basis for $\mathfrak{m}_1$, then $\{v_1,\dots, v_k, iw_1,\dots,iw_l\}$ is a $g'$-orthonormal basis for $\mathfrak{m}'$. The identification $\sigma\: (\mathfrak{m},g)\to(\mathfrak{m}',g')$ given by $v_t\mapsto v_t$ and $w_t\mapsto iw_t$ is an $H$-equivariant isometry, i.e.\@ $\Ad'(h) = \sigma\circ \Ad(h) \circ \sigma^{-1} $ for all $h\in H$. Noting that the isometry $\sigma$ naturally extends to an isomorphism $\sigma\: \Spin(\mathfrak{m})\to \Spin(\mathfrak{m}')$ (by viewing these inside the respective Clifford algebras), we claim that the desired lift is given by $\widetilde{\Ad'}(h) := \sigma( \widetilde{\Ad}(h))$ for all $h\in H$. Denoting by $\lambda$ (resp. $\lambda'$) the covering map $\Spin(\mathfrak{m})\to\SO(\mathfrak{m})$ (resp. $\Spin(\mathfrak{m}')\to\SO(\mathfrak{m}')$), this follows from the calculation
	\begin{align*}
		\lambda'(\widetilde{\Ad'}(h))(v) &= \lambda'(\sigma(\widetilde{\Ad}(h)))(v) =   \sigma(\widetilde{\Ad}(h)) \cdot v \cdot \sigma(\widetilde{\Ad}(h))^{-1} = \sigma(\widetilde{\Ad}(h) \cdot \sigma^{-1}(v)\cdot \widetilde{\Ad}(h)^{-1}) \\
		&= \sigma\big( \lambda(\widetilde{\Ad}(h))(\sigma^{-1}(v))\big) = \sigma\big( \Ad(h)(\sigma^{-1}(v)) \big) = \Ad'(h)(v)
	\end{align*}
	for all $v\in \mathfrak{m}'$. 
\end{proof}
\subsection{Homogeneous $\tad$ Dual Pairs}

From this point forward we restrict attention to the case of generalized $\ts$ data, corresponding to extended symmetric data with $\mathfrak{k}=\mathfrak{sp}(1)$ satisfying the additional condition (iii) in Definition \ref{Chap_dual:3sasakian_data_definition}.
\begin{definition}
	The dual of a $\tad$ homogeneous space $(M=G/H,g,\xi_i,\eta_i,\varphi_i)$ is the $3$-$(\alpha',\delta')$-Sasaki space obtained by applying the duality construction to the corresponding generalized $\ts$ data, where $\alpha':=\alpha$ and $\delta':=-\delta$. The full list of $\tad$ dual pairs can be found in \cite[Table 1]{hom3alphadelta}.
\end{definition}
\begin{remark}
It is easy to check that Proposition \ref{Chap_dual:LCNomizu} is compatible with (\ref{LC_Nomizu_map_explicit_formula}) in the $\tad$ setting:
\begin{align*}
	\Uplambda^{g'}(x_1)(x_2) &=  \tau(\Uplambda^g(x_1))x_2 = \Uplambda^g(x_1)x_2 = \frac{1}{2}[x_1,x_2] , \\
	\Uplambda^{g'}(x)(iy)&=-\tau(\Uplambda^g(x))(iy) +2i[x,y] = -i\Uplambda^g(x)y +2i[x,y] \\
	& = -i(1-\frac{\alpha}{\delta})[x,y] +2i[x,y] =(1-\frac{\alpha'}{\delta'})[x,iy] , \\
	\Uplambda^{g'}(iy)(x)&= -\tau(i\Uplambda^g(y))x =-i\Uplambda^g(y)x = -i \frac{\alpha}{\delta}[y,x] = \frac{\alpha'}{\delta'} [iy,x] , \\
	\Uplambda^{g'}(iy_1)(iy_2)&= -\tau(i\Uplambda^g(y_1))(iy_2) = -\Uplambda^g(y_1)(y_2) = -\frac{1}{2}\text{proj}_{\mathfrak{sp}(1)} [y_1,y_2]=\frac{1}{2}\text{proj}_{\mathfrak{sp}(1)} [iy_1,iy_2].
\end{align*}
\end{remark}
We can also easily describe the relationship between the canonical connections of a $\tad$ dual pair using the identification $\tau$ introduced in the previous section:
\begin{proposition}\label{Chap_dual:CanNomizu}
	If $(M,g)$ and $(M',g')$ are a dual pair of homogeneous $\tad$ spaces then the canonical connections  $\Uplambda\: \mathfrak{m}\times \mathfrak{m}\to \mathfrak{m}$ and $\Uplambda'\: \mathfrak{m}'\times \mathfrak{m}'\to\mathfrak{m}'$ are related by
	\[
	\Uplambda'(V)W=
	\begin{cases}
		\tau(\Uplambda(V))W -\frac{4\alpha'}{\delta'}[V,W]& V \in \mathfrak{sp}(1),\\
		0 & V\in i\mathfrak{m}_1.
	\end{cases}
	\] 
\end{proposition}
\begin{proof}
	Noting that $M$ and $M'$ fiber over Wolf spaces, the result then follows from (\ref{canonical_Nomizu_map_explicit_formula}) by calculating
	\[
	\frac{\delta'-2\alpha'}{\delta'} [V,W]= \frac{\delta+2\alpha}{\delta}[V,W] = \frac{\delta-2\alpha}{\delta}[V,W] +\frac{4\alpha}{\delta}[V,W]=\tau(\Uplambda(V))W -\frac{4\alpha'}{\delta'}[V,W]  
	\]
	for $V\in\mathfrak{sp}(1)$.
\end{proof}
Inspired by \cite[Thm.\@ 7.2]{kath_Tduals}, we have the following theorem relating the $\mathcal{H}$-Killing spinors on a compact homogeneous $\tad$ space and its non-compact dual. The special case of dimension $7$ admits some nice simplifications, which will be discussed in more detail in the subsequent section.
\begin{theorem}\label{Chap_dual:spinorialdualitytheorem}
	Let $(M,g)$ and $(M',g')$ be a dual pair of homogeneous $\tad$ spaces of dimension $4n-1$, and identify the spinor modules $\Sigma\cong \Sigma'\cong  \Lambda^{\bullet} \C^{2n-1}$. If $\psi\: G\to \Sigma$, $\psi\equiv u$ is a constant $H$-equivariant map then the constant map $\psi'\: G' \to \Sigma'$, $\psi'\equiv u$ is also $H$-equivariant. In this case:
	
	\begin{enumerate}[(i)]
		\item The invariant spinor on $(M,g)$ induced by $\psi$ satisfies the $\mathcal{H}$-Killing equation if and only if the invariant spinor on $(M',g')$ induced by $\psi$ satisfies the $\mathcal{H}'$-Killing equation\footnote{By this we mean the equation $\nabla^{g'}_X\psi' = \frac{\alpha'}{2} X\cdot \psi' +\frac{\alpha'-\delta'}{2} \sum_{i=1}^3 \eta_i'(X)\Phi_i'\cdot \psi'$, where $\alpha':=\alpha$, $\delta':=-\delta$.};
		\item If the invariant spinor on $(M,g)$ induced by $\psi$ is $\nabla$-parallel, then the invariant spinor on $(M',g')$ satisfies
		\begin{align} \label{Chap_dual:canonicaldualequation}
			\nabla_X' \psi' = 	\begin{cases}
				2\alpha' (\Phi_i'  -  \xi_j'\cdot\xi_k')\cdot  \psi' & X =\xi_i'\in\mathcal{V}', \\
				0 & X\in \mathcal{H}',
			\end{cases}  
		\end{align}
		for any even permutation $(i,j,k)$ of $(1,2,3)$. 
	\end{enumerate} 
\end{theorem}
\begin{proof}
	We would like to define an isometry $\theta\: (\mathfrak{m},g)\to (\mathfrak{m}',g')$ in order to compare the two spaces. The obvious choice is the identification $x\mapsto x$, $y\mapsto  iy$ used in the proof of Proposition \ref{Chap_dual:dualspinlift}, however this leads to a problem with the orientations of the two spaces. Indeed, using the notation of \cite[Thm.\@ 3.1.1]{hom3alphadelta}, the Reeb vector fields of the dual pair are related by
	\[
	\xi_i'=\delta' \sigma_i = -\delta\sigma_i=-\xi_i,
	\]
	so an orthonormal frame 
	$
	\{\xi_1,\xi_2,\xi_3, e_4,\dots,e_{4n-4} \}
	$ inside the standard orientation for $(M,g)$ would be identified with the frame 
	$
	\{-\xi_1'-\xi_2',-\xi_3', ie_4,\dots ie_{4n-4}\},
	$ which is \emph{not} oriented in the standard way for the dual $3$-$(\alpha',\delta')$-Sasaki space $(M',g')$. Thus we choose instead the isometry 
	\[\theta := -\Id\rvert_{\mathfrak{sp}(1)} \oplus i\Id\rvert_{\mathfrak{m}_1},
	\]
	which identifies $\xi_i$ with $\xi_i'$. Then, under identification of the spinor modules as above, each tangent vector $V\in\mathfrak{m}$ and its image $\theta(V)\in \mathfrak{m}'$ act as the same operator by Clifford multiplication (likewise for the extension of $\theta$ to tensors and differential forms). Under this identification it is clear that $\psi$ is $H$-invariant if and only if $\psi'$ is $H$-invariant, due to the fact that the $2$-form $\ad'(h) \in \mathfrak{so}(\mathfrak{m}',g')$ coincides with $\theta(\ad(h))$ for all $h \in \mathfrak{h}$. By \cite[Thm.\@ 3.1.1]{hom3alphadelta} we have
	\[
	\varphi_i' = \frac{1}{2\delta'}\ad(\xi_i')\rvert_{\mathfrak{sp}(1)} + \frac{1}{\delta'}\ad(\xi_i')\rvert_{i\mathfrak{m}_1} =   \frac{1}{2\delta}\ad(\xi_i)\rvert_{\mathfrak{sp}(1)} + \frac{1}{\delta}\ad(\xi_i)\rvert_{i\mathfrak{m}_1}  , 
	\]	
	and therefore $\Phi_i' = \theta(\Phi_i)$. 
	In the horizontal directions, it follows from Lemma \ref{Chap_dual:endomorphismparts} and Proposition \ref{Chap_dual:LCNomizu} above, together with the explicit formula for the Nomizu map recalled in (\ref{LC_Nomizu_map_explicit_formula}), that
	\begin{align*} \theta^{-1}( \Uplambda^{g'}(iy) )&= \theta^{-1} (-\tau(i\Uplambda^g(y)))    =  \theta^{-1}(- \frac{\alpha}{\delta} \ad(iy)\rvert_{\mathfrak{sp}(1)}  +\frac{1}{2} \text{proj}_{\mathfrak{sp}(1)} \ad(iy)\rvert_{i\mathfrak{m}_1})\\
		&= \frac{\alpha}{\delta} \ad(y)\rvert_{\mathfrak{sp}(1)} +\frac{1}{2}\text{proj}_{\mathfrak{sp}(1)} \ad(y)\rvert_{\mathfrak{m}_1} \\
		&= \Uplambda^g(y).
	\end{align*}
	Thus if $\psi\equiv u$ is an invariant $\mathcal{H}$-Killing spinor, then we have
	\[
	\widetilde{\Uplambda}^{g'}(iy)\cdot u  = \theta(\widetilde{\Uplambda}^g(y))\cdot u = \theta(\frac{\alpha}{2}y) \cdot u = \frac{\alpha'}{2} (iy)\cdot u , 
	\] where $\widetilde{\Uplambda^{g}}$, $\widetilde{\Uplambda^{g'}}$ denote the spin lifts of $\Uplambda^g$, $\Uplambda^{g'}$. The situation in the vertical directions is somewhat more complicated. Arguing as above, we have{\small
		\begin{align*}
			\theta^{-1}(\Uplambda^{g'}(\xi_i)) &=\theta^{-1}( \tau(\Uplambda^g(\xi_i))\rvert_{\mathfrak{sp}(1)} -\tau(\Uplambda^g(\xi_i))\rvert_{i\mathfrak{m}_1} + 2\ad(\xi_i)\rvert_{i\mathfrak{m}_1} ) \\
			&= \Uplambda^g(\xi_i)\rvert_{\mathfrak{sp}(1)} -\Uplambda^g(\xi_i)\rvert_{\mathfrak{m}_1} + 2\ad(\xi_i)\rvert_{\mathfrak{m}_1}
			=  \Uplambda^g(\xi_i)\rvert_{\mathfrak{sp}(1)} - \Uplambda^g(\xi_i) +\Uplambda^g(\xi_i)\rvert_{\mathfrak{sp}(1)} + 2\ad(\xi_i)\rvert_{\mathfrak{m}_1}\\ &= 2\Uplambda^g(\xi_i)\rvert_{\mathfrak{sp}(1)} -\Uplambda^g(\xi_i) + 2\ad(\xi_i)\rvert_{\mathfrak{m}_1} 
			= \ad(\xi_i)\rvert_{\mathfrak{sp}(1)} - \Uplambda^g(\xi_i) +2\ad(\xi_i)\rvert_{\mathfrak{m}_1}\\
			&= - 2\delta \Phi_i -\Uplambda^g(\xi_i)
		\end{align*}
	}(where for the final equality we have identified the endomorphism field $\ad(\xi_i)\rvert_{\mathfrak{sp}(1)}$ with the $2$-form $-2\delta \Phi_i\rvert_{\mathcal{V}}$ using the metric $g$). Taking the spin lift then gives 
	\[
	\theta^{-1}(\widetilde{\Uplambda}^{g'}(\xi_i)) = -\delta\Phi_i -\widetilde{\Uplambda}^g(\xi_i) , 
	\] or equivalently, 
	\[
	\theta^{-1}(\widetilde{\Uplambda}^{g'}(\xi_i')) = - \delta' \Phi_i + \widetilde{\Uplambda}^g(\xi_i) , 
	\]
	and hence
	\begin{align*}
		\widetilde{\Uplambda^{g'}}(\xi_i') \cdot u &= \theta(-\delta'\Phi_i + \widetilde{\Uplambda}^g(\xi_i)) \cdot u  =-\delta'\Phi_i'\cdot u + \theta(\widetilde{\Uplambda}^g(\xi_i)) \cdot u \\
		&=-\delta' \Phi_i'\cdot u +\theta( \frac{\alpha}{2}\xi_i +\frac{\alpha-\delta}{2}\Phi_i)\cdot u 
		= -\delta' \Phi_i'\cdot u + \frac{\alpha'}{2} \xi_i' \cdot u +\frac{\alpha'+\delta'}{2} \Phi_i'\cdot u \\
		&= \frac{\alpha'}{2}\xi_i'\cdot u +\frac{\alpha'-\delta'}{2}\Phi_i'\cdot u . 
	\end{align*}
	It follows that 
	\[
	\widetilde{\Uplambda}^{g'}(V)\cdot u = \frac{\alpha'}{2}V\cdot u + \frac{\alpha'-\delta'}{2}\sum_{i=1}^3 \eta_i'(V)\Phi_i'\cdot u  , 
	\]for all $V\in\mathfrak{m}'$, proving the first part of the theorem. Assume now that $\psi\equiv u$ is parallel for the canonical connection. It is immediately apparent from Proposition \ref{Chap_dual:CanNomizu} that $\widetilde{\Uplambda'}(iy)\cdot u =0$ for all $iy\in i\mathfrak{m}_1$, and in the vertical directions
	\begin{align*}
		\theta^{-1}(\Uplambda'(\xi_i')) &= \theta^{-1}\left(\tau(\Uplambda(\xi_i'))-\frac{4\alpha'}{\delta'}\ad(\xi_i')\right) = \Uplambda(\xi_i') -\frac{4\alpha'}{\delta'} \ad(\xi_i') \\
		&= -\Uplambda(\xi_i) - 4\alpha' ( \varphi_i + \varphi_i\rvert_{\mathfrak{sp}(1)}) = -\Uplambda(\xi_i) +4\alpha'( \Phi_i + \Phi_i^{\mathcal{V}}   ) 
	\end{align*}
	(where for the final equality we have identified skew-symmetric endomorphisms with differential forms using the metric $g'$). Applying the spin lift of this operator to $\psi\equiv u$ gives 
	\[
	\widetilde{\Uplambda'}(\xi_i') \cdot u = 0+ 2\alpha' \Phi_i'\cdot u -2\alpha' \xi_j' \cdot \xi_k' \cdot u , 
	\]
	as desired.
\end{proof}
\begin{remark}
	The preceding theorem parallels Kath's result \cite[Thm.\@ 7.2]{kath_Tduals}, which showed that real Killing spinors on compact homogeneous Riemannian spin manifolds $(M=G/H,g)$ correspond to real Killing spinors (for the same Killing number) on pseudo-Riemannian homogeneous spaces strongly $T$-dual to $M$. This suggests that our $\mathcal{H}$-Killing equation (\ref{deformedKillingspinorsbundle}) is a natural generalization of Killing spinors to the $\tad$ setting (a priori, it is possible for the same spinor to satisfy many different equations). Note also that it doesn't make sense to consider spinors dual to the deformed Killing spinors (\ref{KS_with_torsion}) due to the problem with the coefficients $\sqrt{\alpha\delta}$ which we have already encountered.
\end{remark}
\section{The Special Case of Dimension 7}\label{section:dim7}
In this section we examine more closely the situation in dimension 7, for both homogeneous and non-homogeneous $\tad$ manifolds. This dimension is special due to the well-known correspondence between unit spinors (up to sign) and $\G_2$-structures described in \cite{dim67}, allowing the matter of special spinors to be approached using well-known results from $\G_2$-geometry. In particular, one can exploit the fact that the canonical connection of the $\tad$ structure is the characteristic connection of the $\G_2$-structure induced by a certain spinor with desirable geometric properties (the so-called \emph{canonical spinor}) \cite[\S4.5]{3str}. Concretely, for any 7-dimensional $\tad$ manifold $(M^7,g,\varphi_i, \xi_i,\eta_i )$ (not necessarily homogeneous), it is shown in \cite[Thm.\@ 4.5.1]{3str} that the canonical connection $\nabla$ arises as the characteristic connection of the co-calibrated $\G_2$-structure 
\begin{align}\label{G2form}
\omega := \eta_{1,2,3} + \sum_{i=1}^3 \eta_i\wedge \Phi_i^{\mathcal{H}}
\end{align}
defined naturally in terms of the $\tad$ structure tensors. The corresponding spinor $\psi_0$ is called the \emph{canonical spinor}, and can be characterized, as in \cite[Def.\@ 4.5.1]{3str}, as the unique spinor (up to sign) such that
\[ \nabla \psi_0=0,\qquad \omega\cdot \psi_0=-7\psi_0, \qquad |\psi_0|=1.   \]
The three \emph{auxiliary spinors} $\psi_i$, $i=1,2,3$ are defined as the Clifford products of the Reeb vector fields with $\psi_0$,
\[
\psi_1:= \xi_1\cdot \psi_0,\qquad \psi_2:=\xi_2\cdot \psi_0,\qquad \psi_3:=\xi_3\cdot \psi_0.
\]
In the $\ts$ setting ($\alpha=\delta=1$) the canonical spinor is a generalized Killing spinor, and the auxiliary spinors are the three linearly independent Killing spinors known to be carried by a $\ts$ manifold in dimension $7$ \cite{3Sasdim7}. For arbitrary $\alpha,\delta$, we recall that the canonical and auxiliary spinors are generalized Killing spinors for the Levi-Civita connection:  
\begin{theorem}\label{3adGKSTheorem}
	\emph{(\cite[Thm.\@ 4.5.2]{3str}).} The canonical and auxiliary spinors are Riemannian generalized Killing spinors,
	\begin{align} \label{dim7spinorialeqn3ad}
	\nabla^g_X\psi_0 &=  \begin{cases} 
		\frac{2\alpha-\delta}{2}X\cdot \psi_0 & X\in\mathcal{V} \\
	-\frac{3\alpha}{2}X\cdot \psi_0  &  X\in\mathcal{H}, \\
	\end{cases} \qquad 
	\nabla^g_X\psi_i = \begin{cases} 
	\frac{2\alpha-\delta}{2} \xi_i\cdot \psi_i &  X=\xi_i \\
	\frac{3\delta-2\alpha}{2}\xi_j\cdot \psi_i  & X=\xi_j \ (j\neq i) \\
	\frac{\alpha}{2}X\cdot \psi_i & X\in\mathcal{H}
	\end{cases}
	\end{align}
	for $i,j=1,2,3$.
\end{theorem}
We refer the reader to \cite[Ex.\@ 4.18]{AHLspheres} for a detailed discussion of the homogeneous example $S^{7}=\Sp(2)/\Sp(1)$, including explicit calculations of the invariant spinors and their spinorial field equations. In the following theorem we show that the canonical and auxiliary spinors on a $7$-dimensional homogeneous $\tad$ space correspond in a $1$-to-$1$ manner with their dual counterparts:
\begin{theorem}\label{Chap:duality:dual_equations_theorem_hom3ad}
	Suppose that $(M=G/H,g)$ is a compact simply-connected 7-dimensional homogeneous $\tad$ space, and $(M'=G'/H, g')$ its non-compact dual. Under the identification of spinor bundles $\Sigma\cong \Sigma'$ as in Theorem \ref{Chap_dual:spinorialdualitytheorem}, the canonical and auxiliary spinors $\psi_i$, $i=0,1,2,3$ on $(M,g)$ are given by constant $H$-equivariant maps $G\to \Sigma $, and the corresponding spinors $\psi_i' \in \Sigma'$, $i=0,1,2,3$ are the canonical and auxiliary spinors of the dual $3$-$(\alpha',\delta')$-Sasaki structure. In particular they are Riemannian generalized Killing spinors, satisfying:
	\begin{align} \label{Chap_dual:dim7dualspinorialeqn3ad}
		\nabla^{g'}_X\psi_0' &=  \begin{cases} 
			\frac{2\alpha'-\delta'}{2} X\cdot \psi_0' & X\in\mathcal{V}', \\
			-\frac{3\alpha'}{2}X\cdot \psi_0'  &  X\in\mathcal{H}', \\
		\end{cases} \qquad 
		\nabla^{g'}_X\psi_i' = \begin{cases} 
			\frac{2\alpha'-\delta'}{2}\xi_i'\cdot \psi_i' &  X=\xi_i' , \\
			\frac{3\delta'-2\alpha'}{2} \xi_j'\cdot \psi_i'  & X=\xi_j'  \ (j\neq i), \\
			\frac{\alpha'}{2} X\cdot \psi_i' & X\in\mathcal{H}',
		\end{cases}
	\end{align}
	for $i,j=1,2,3$, where $\mathcal{V}'$, $\mathcal{H}'$ denote the vertical and horizontal bundles respectively of $M'$.
\end{theorem}
\begin{proof}
	Since $M$ is compact ($\alpha\delta>0$), the generalized 3-Sasakian data determining it also determines a 7-dimensional compact homogeneous 3-Sasakian space. From the classification of homogeneous 3-Sasakian spaces in \cite[Thm.\@ C]{BG3Sas} it then follows that $M$ is either
	\[
	S^{7} =\frac{\Sp(2)}{\Sp(1)} \quad \text{ or } \quad \frac{\SU(3)}{S(\U(1)\times \U(1))}
	\]
	(the case $\RP^7$ is excluded by the assumption that $M$ is simply-connected), so in particular $G=\Sp(2)$ or $\SU(3)$. To see that $\psi_0$ is $G$-invariant, we note that it is determined up to sign as a unit length element of the 1-dimensional $(-7)$-eigenspace for the action of $\omega$ on $\Sigma$; then, using $G$-invariance of $\omega$, we calculate at the origin
	\[  
	(\omega\cdot (g_0\psi_0) )(g)= (\omega\cdot \psi_0) (g_0^{-1}g) = -7\psi_0(g_0^{-1}g) = -7(g_0\psi_0)(g), \qquad \text{ for all } g_0,g\in G.
	\]
	This shows that the 1-dimensional space $\C\psi_0 $ is a $G$-subrepresentation of $\Sigma$, and this subrepresentation must be trivial since $G=\Sp(2)$ or $\SU(3)$ (i.e. $\psi_0$ corresponds to a constant map $G\to \Sigma$). As the Reeb vector fields $\xi_i$ are invariant, we then have that the auxiliary spinors $\psi_i=\xi_i\cdot \psi_0$, $i=1,2,3$ are also invariant.
	%
	Next, we observe that the canonical $\G_2$-form $\omega'$ of $(M',g')$, defined using the $3$-$(\alpha',\delta')$-Sasaki structure tensors analogously to (\ref{G2form}), coincides with $\theta(\omega)$, and hence its $(-7)$-eigenspace inside $\Sigma'$ coincides with the span of $\psi_0$. The auxiliary spinors of $(M',g')$ therefore coincide with $\psi_i$, $i=1,2,3$, and the spinorial equations (\ref{Chap_dual:dim7dualspinorialeqn3ad}) follow from Theorem \ref{3adGKSTheorem} applied to $M'$.
	%
	%
	%
	%
	%
\end{proof}
For any $7$-dimensional $\tad$ manifold (not necessarily homogeneous), we can also compare the spinorial equation satisfied by the auxiliary spinors (the second equation in (\ref{dim7spinorialeqn3ad})) to the $\mathcal{H}$-Killing equation (\ref{deformedKillingspinorsbundle}). In order to accomplish this we have the following lemma, which may be proved by a straightforward calculation in the spin representation:
\begin{lemma}\label{Chap_dual:Cliffproductslemma}
	If $(M^7,g,\xi_i,\eta_i,\varphi_i)$ is a 7-dimensional $\tad$ manifold, the canonical and auxiliary spinors satisfy
	%
	%
	%
	%
	\[
	\Phi_i\cdot \psi_0 = \psi_i, \quad \Phi_i\cdot \psi_i = \xi_i\cdot \psi_i,\quad \Phi_i \cdot \psi_j = -3 \xi_i\cdot \psi_j,
	\]	
	for all $i,j=1,2,3$ with $i\neq j$. Furthermore, for any even permutation $(i,j,k)$ of $(1,2,3)$, the canonical spinor satisfies 
	\begin{align}\label{Chap_dual:canonicalCliffrelation}
		(\Phi_i - \xi_j\cdot \xi_k)\cdot \psi_0 = 0.
	\end{align}
\end{lemma}
As an immediate consequence of Lemma \ref{Chap_dual:Cliffproductslemma} we obtain following theorem, which asserts that the $\mathcal{H}$-Killing equation is equivalent to the equation satisfied by the auxiliary spinors in dimension $7$. This is further evidence that $\mathcal{H}$-Killing spinors are the natural generalization of Killing spinors to the $\tad$ setting.
\begin{theorem}\label{Chap_duality:dim7_equivalence_of_two_eqns}
	If $(M^7,g,\xi_i,\eta_i,\varphi_i)$ is a 7-dimensional $\tad$ manifold, then a spinor $\psi\in\Gamma(E)$ satisfies the $\mathcal{H}$-Killing equation (\ref{deformedKillingspinorsbundle}) if and only if it satisfies the second equation in (\ref{dim7spinorialeqn3ad}).
\end{theorem}
\begin{proof}
	Using Lemma \ref{Chap_dual:Cliffproductslemma}, we calculate{\small
		\begin{align*}
			\frac{\alpha}{2}X\cdot \psi_i  + \frac{\alpha-\delta}{2} \sum_{p=1}^3\eta_p(X)\Phi_p\cdot \psi_i &= \frac{\alpha}{2} X\cdot \psi_i + \frac{\alpha-\delta}{2}\left( \eta_i(X) \xi_i -3\eta_{i+1}(X) \xi_{i+1} -3 \eta_{i+2}(X) \xi_{i+2}  \right)\cdot \psi_i ,
		\end{align*}
	}where the indices $i,i+1,i+2$ are taken modulo 3. The result then follows by substituting the various cases for $X$ from (\ref{dim7spinorialeqn3ad}). 
\end{proof}

\begin{remark}
	Substituting (\ref{Chap_dual:canonicalCliffrelation}) into (\ref{Chap_dual:canonicaldualequation}), one immediately sees that the dual of the canonical spinor on a $7$-dimensional homogeneous $\tad$ spaces is again $\nabla'$-parallel. This is a manifestation of the correspondence, proved in Theorem \ref{Chap:duality:dual_equations_theorem_hom3ad}, of the canonical spinors on dual pairs of $7$-dimensoinal $\tad$ homogeneous spaces.
\end{remark}

\begin{remark}
Lemma \ref{Chap_dual:Cliffproductslemma} also permits us to simplify the Riemannian Dirac equation satisfied by $\mathcal{H}$-Killing spinors in dimension $7$. By Theorem \ref{Chap_duality:dim7_equivalence_of_two_eqns}, the space of $\mathcal{H}$-Killing spinors in this dimension contains the $\R$-span of the auxiliary spinors $\psi_i$, $i=1,2,3$. As an aside, we recall from \cite[Rmk.\@ 4.19]{AHLspheres} that the canonical and auxiliary spinors are given, with respect to an adapted local frame, by
\[ \psi_0 =\frac{1}{\sqrt{2}} (\omega+iy_1),\quad \psi_1 = \frac{1}{\sqrt{2}} (i\omega+y_1) ,\quad \psi_2 = \frac{1}{\sqrt{2}}(-1+iy_1\wedge \omega) ,\quad \psi_3 = \frac{1}{\sqrt{2}}(-i + y_1\wedge \omega),   \]
where $\omega:= y_2\wedge y_3$, and comparing with Table \ref{Tab:spinors_PsiEi_low_dim} shows that the auxiliary spinors give a basis of $E=E_1+E_2+E_3$ consisting of $\mathcal{H}$-Killing spinors\footnote{A word of caution: the auxiliary spinor $\psi_i$ is not necessarily a section of $E_i$, for example $\psi_1$ is a section of $E_2$ and $E_3$ but not $E_1$.}. Now, letting $\psi$ be an $\mathcal{H}$-Killing spinor, it follows from Remark \ref{Dirac_remark} that $\psi$ is a Dirac eigenspinor for all values of $\alpha,\delta$:
\[
 D\psi   = \frac{-2\alpha -5\delta}{2} \psi .
\]
Similarly, the canonical spinor is a Dirac eigenspinor for all $\alpha,\delta$:
\[
D\psi_0 = \frac{6\alpha +3\delta}{2}\psi_0 .
\]
By \cite[Prop.\@ 2.3.3]{3str}, the Riemannian Ricci curvature of a $7$-dimensional $\tad$ manifold is given by
\[ \text{Ric}^g =  2\alpha( 6\delta - 3\alpha) g + 2(\alpha-\delta) (5\alpha-\delta) \sum_{=1}^3 \eta_i\otimes \eta_i ,  \]
hence the scalar curvature is (constant and) equal to $R_0 := 6(\delta^2+8\alpha\delta-2\alpha^2)$. One checks that the eigenspinor $\psi$ (resp.\@ $\psi_0$) realizes Friedrich's bound $\pm \frac{1}{2} \sqrt{\frac{7R_0}{6}}$ for the smallest eigenvalue of the Dirac operator (see e.g.\@ \cite[Chap.\@ 1.5]{BFGK}) if and only if $\delta=\alpha$ (resp.\@ $\delta=5\alpha$). Note that these are precisely the values of $\alpha,\delta$ for which the auxiliary spinors $\psi_i$ (resp.\@ the canonical spinor $\psi_0$) are Killing spinors \cite[Thm.\@ 4.5.2]{3str}.  
\end{remark}

\begin{remark}
Finally, we comment on $\nabla$-parallel spinors in dimension $7$. There is always at least one such spinor, namely the canonical spinor, which is $\nabla$-parallel independent of the values of $\alpha,\delta$. In the case of a parallel $\tad$ structure ($\beta=0$), the holonomy algebra of $\nabla$ is contained in $\mathfrak{sp}(1)\subset \mathfrak{so}(4)\oplus \mathfrak{so}(3)$ (see \cite[Cor.\@ 4.1.2]{3str}), and it follows that the spin lift of the infinitesimal holonomy representation annihilates at least $4$ vectors in the spin representation (cf.\@ \cite[\S4]{AHLspheres}). On a simply-connected manifold, this implies that the holonomy group $\text{Hol}(\nabla)$ stabilizes at least $4$ vectors in the spin representation, hence the manifold carries at least $4$ linearly independent (globally-defined) $\nabla$-parallel spinors. Since the Reeb vector fields are parallel when $\beta=0$ (see Proposition \ref{prelims:canonical_connection}), these four $\nabla$-parallel spinor fields are explicitly given by the canonical and auxiliary spinors. The holonomy algebra $\mathfrak{hol}(\nabla)$ is unstable under perturbations of $\beta$, in the sense that for $\beta\neq 0$ the auxiliary spinors are no longer parallel but rather satisfy the equation
\[ \nabla_X \psi_i = \nabla_X (\xi_i \cdot \psi_0)=  \beta (\eta_k(X)\xi_j - \eta_j(X)\xi_k)\cdot \psi_0 \]
for all $X \in TM$. Said differently, when $\beta\neq 0$ the holonomy algebra is no longer contained in $\mathfrak{sp}(1)$:
\[  \mathfrak{sp}(1)\subsetneq \mathfrak{hol}(\nabla) \subseteq \mathfrak{g}_2 \subseteq \mathfrak{so}(7) \]
(it remains a subalgebra of $\mathfrak{g}_2$ since it still stabilizes $\psi_0$).
\end{remark}

\bibliography{bibliodatabase}

\begin{thebibliography}{BFGK91}

\bibitem[ACFH15]{dim67}
Ilka Agricola, Simon~G. Chiossi, Thomas Friedrich, and Jos Höll.
\newblock Spinorial description of $\text{SU}(3)$- and $\text{G}_2$-manifolds.
\newblock {\em Journal of Geometry and Physics}, 98:535–555, Dec 2015.

\bibitem[AD20]{3str}
Ilka Agricola and Giulia Dileo.
\newblock Generalizations of 3-{S}asakian manifolds and skew torsion.
\newblock {\em Advances in Geometry}, 20(3):331--374, 2020.

\bibitem[ADS21]{hom3alphadelta}
Ilka Agricola, Giulia Dileo, and Leander Stecker.
\newblock Homogeneous non-degenerate 3-$(\alpha,\delta)$-{S}asaki manifolds and
  submersions over quaternionic {K}ähler spaces.
\newblock {\em Annals of Global Analysis and Geometry}, 60(1):111–141, Apr
  2021.

\bibitem[ADS23]{ADScurv}
Ilka Agricola, Giulia Dileo, and Leander Stecker.
\newblock Curvature properties of 3-$(\alpha ,\delta )$-{Sasaki} manifolds.
\newblock {\em Annali di Matematica Pura ed Applicata (1923 -)}, Mar 2023.

\bibitem[AF10]{3Sasdim7}
Ilka Agricola and Thomas Friedrich.
\newblock 3-{S}asakian manifolds in dimension seven, their spinors and
  ${G}_2$-structures.
\newblock {\em Journal of Geometry and Physics}, 60(2):326–332, Feb 2010.

\bibitem[Agr06]{SRNI}
Ilka Agricola.
\newblock The {Srn{\'{\i}}} lectures on non-integrable geometries with torsion.
\newblock {\em Arch. Math., Brno}, 42(5):5--84, 2006.

\bibitem[AHL23]{AHLspheres}
Ilka Agricola, Jordan Hofmann, and Marie-Amélie Lawn.
\newblock Invariant spinors on homogeneous spheres.
\newblock {\em Differential Geometry and its Applications}, 89:102014, 2023.

\bibitem[ANT23]{ANT_book_principal_fibre_bundles}
Ilka Agricola, Henrik Naujoks, and Marvin Theiss.
\newblock {\em Geometry of Principal Fibre Bundles (to appear)}.
\newblock 2023.

\bibitem[B{\"a}r93]{Bar}
Christian B{\"a}r.
\newblock Real {K}illing spinors and holonomy.
\newblock {\em Communications in Mathematical Physics}, 154:509--521, 1993.

\bibitem[BB12]{Julia_BB_PhD_thesis}
Julia Becker-Bender.
\newblock {\em Dirac-Operatoren und Killing-Spinoren mit Torsion}.
\newblock PhD thesis, Philipps-Universit\"{a}t Marburg, 2012.

\bibitem[Ber55]{Berger_holonomy}
Marcel Berger.
\newblock Sur les groupes d'holonomie homog{\`e}ne des vari{\'e}t{\'e}s {\`a}
  connexion affine et des vari{\'e}t{\'e}s riemanniennes.
\newblock {\em Bull. Soc. Math. Fr.}, 83:279--330, 1955.

\bibitem[Bes87]{Besse_Einstein_manifolds}
Arthur~L. Besse.
\newblock {\em Einstein manifolds}, volume~10 of {\em Ergeb. Math. Grenzgeb.,
  3. Folge}.
\newblock Springer, Cham, 1987.

\bibitem[BFGK91]{BFGK}
Helga Baum, Thomas Friedrich, Ralf Grunewald, and Ines Kath.
\newblock {\em Twistors and {K}illing Spinors on {R}iemannian Manifolds}.
\newblock B. G. Teubner Verlagsgesellschaft, 1991.

\bibitem[BG99]{BG_3Sas_paper}
Charles Boyer and Krzysztof Galicki.
\newblock 3-{Sasakian} manifolds.
\newblock In {\em Surveys in differential geometry. Vol. VI: Essays on Einstein
  manifolds. Lectures on geometry and topology, sponsored by Lehigh
  University's Journal of Differential Geometry}, pages 123--184. Cambridge,
  MA: International Press, 1999.

\bibitem[BGM94]{BG3Sas}
Charles~P. Boyer, Krzysztof Galicki, and Benjamin~M. Mann.
\newblock The geometry and topology of 3-{S}asakian manifolds.
\newblock {\em Journal f\"ur die reine und angewandte Mathematik},
  455:183--220, 1994.

\bibitem[CS07]{hypo}
Diego Conti and Simon Salamon.
\newblock Generalized {K}illing spinors in dimension 5.
\newblock {\em Transactions of the American Mathematical Society},
  359(11):5319--5343, May 2007.

\bibitem[DKL22]{invariantspinstructures}
Jordi {Daura Serrano}, Michael Kohn, and Marie-Am{\'e}lie Lawn.
\newblock G-invariant spin structures on spheres.
\newblock {\em Ann. Global Anal. Geom.}, 62(2):437--455, 2022.

\bibitem[DOP20]{homdata}
Cristina Draper, Miguel Ortega, and Francisco~J. Palomo.
\newblock Affine connections on 3-{Sasakian} homogeneous manifolds.
\newblock {\em Math. Z.}, 294(1-2):817--868, 2020.

\bibitem[FK90]{Fried90}
Thomas Friedrich and Ines Kath.
\newblock 7-dimensional compact {R}iemannian manifolds with {K}illing spinors.
\newblock {\em Communications in Mathematical Physics}, 133:543--561, 1990.

\bibitem[FKMS97]{nearly_parallel_g2}
Thomas Friedrich, Ines Kath, Andrei Moroianu, and Uwe Semmelmann.
\newblock On nearly parallel {{\(G_2\)}}-structures.
\newblock {\em J. Geom. Phys.}, 23(3-4):259--286, 1997.

\bibitem[{Fri}00]{FriedrichBook}
Thomas {Friedrich}.
\newblock {\em {D}irac operators in {R}iemannian geometry}, volume~25.
\newblock Providence, RI: American Mathematical Society (AMS), 2000.

\bibitem[GRS23]{Leander_roots}
Oliver Goertsches, Leon Roschig, and Leander Stecker.
\newblock Revisiting the classification of homogeneous 3-{Sasakian} and
  quaternionic {K{\"a}hler} manifolds.
\newblock {\em Eur. J. Math.}, 9(1):28, 2023.
\newblock Id/No 11.

\bibitem[Hof22]{Hof22}
Jordan Hofmann.
\newblock Homogeneous {S}asakian and 3-{S}asakian structures from the spinorial
  viewpoint, 2022.
\newblock Preprint. \url{https://arxiv.org/abs/2208.09301}.

\bibitem[Iva04]{Spin7}
Stefan Ivanov.
\newblock Connections with torsion, parallel spinors and geometry of {S}pin(7)
  manifolds.
\newblock {\em Mathematical {R}esearch {L}etters}, 11(2):171--186, 2004.

\bibitem[Kat00]{kath_Tduals}
Ines Kath.
\newblock Pseudo-{Riemannian} {{\(T\)}}-duals of compact {Riemannian}
  homogeneous spaces.
\newblock {\em Transform. Groups}, 5(2):157--179, 2000.

\bibitem[KF00]{FrKi00}
Eui~Chul Kim and Thomas Friedrich.
\newblock The {Einstein}-{Dirac} equation on {Riemannian} spin manifolds.
\newblock {\em J. Geom. Phys.}, 33(1-2):128--172, 2000.

\bibitem[KN63]{KN1}
Shoshichi Kobayashi and Katsumi Nomizu.
\newblock {\em Foundations of differential geometry. {I}}, volume~15 of {\em
  Intersci. Tracts Pure Appl. Math.}
\newblock Interscience Publishers, New York, NY, 1963.

\bibitem[KN69]{KN2}
Shoshichi Kobayashi and Katsumi Nomizu.
\newblock {\em Foundations of Differential Geometry, Volume II}.
\newblock Interscience Publishers, 1969.

\bibitem[Kuo70]{3Sas_structure_reduction}
{Ying-yan} Kuo.
\newblock {On almost contact $3$-structure}.
\newblock {\em Tohoku Mathematical Journal}, 22(3):325 -- 332, 1970.

\bibitem[LM89]{LM}
Herbert~Blaine {Lawson, Jr.} and Marie{-}Louise Michelsohn.
\newblock {\em Spin Geometry}.
\newblock Princeton University Press, 1989.

\bibitem[Nom54]{Nomizumap}
Katsumi Nomizu.
\newblock Invariant affine connections on homogeneous spaces.
\newblock {\em American Journal of Mathematics}, 76(1):33--65, 1954.

\bibitem[Sim62]{Simons_holonomy}
James Simons.
\newblock On the transitivity of holonomy systems.
\newblock {\em Ann. Math. (2)}, 76:213--234, 1962.

\bibitem[Wan58]{Wangconnections}
{Hsien{-}chung} Wang.
\newblock On invariant connections over a principal fibre bundle.
\newblock {\em Nagoya Mathematical Journal}, 13:1--19, 1958.

\bibitem[Wan89]{Wang}
Mckenzie Y.~K. Wang.
\newblock Parallel spinors and parallel forms.
\newblock {\em Annals of Global Analysis and Geometry}, 7:59--68, 01 1989.

\end{thebibliography}
\bibliographystyle{alpha}


\end{document}